\documentclass[11pt,amssymb]{amsart}
\usepackage{amssymb}
\usepackage[mathscr]{eucal}

\newcommand{\Z}{{\mathbb Z}}
\newcommand{\C}{{\mathbb C}}
\newcommand{\Q}{{\mathbb Q}}
\newcommand{\R}{{\mathbb R}}

\newcommand{\G}{{\mathbb G}}

\newcommand{\ba}{\mbox{\boldmath{$\alpha$}}}
\newcommand{\bm}{\mbox{\boldmath{$\mu$}}}
\newtheorem{thm}{Theorem}[section]
\newtheorem{lemma}[thm]{Lemma}
\newtheorem{prop}[thm]{Proposition}
\newtheorem{cor}[thm]{Corollary}

\usepackage[all,cmtip]{xy}
\begin{document}
\title[Representations of Chevalley groups]{On linear representations of Chevalley groups over commutative rings}


\begin{abstract}
Let $G$ be the universal Chevalley-Demazure group scheme corresponding to a reduced irreducible root system of rank $\geq 2$, and let $R$ be a commutative ring. We analyze the linear representations $\rho \colon G(R)^+ \to GL_n (K)$ over an algebraically closed field $K$ of the elementary subgroup $G(R)^+ \subset G(R).$ Our main result is that under certain conditions, any such representation has a standard description, i.e. there exists a commutative finite-dimensional $K$-algebra $B$, a ring homomorphism $f \colon R \to B$ with Zariski-dense image, and a morphism of algebraic groups $\sigma \colon G(B) \to GL_n (K)$ such that $\rho$ coincides with $\sigma \circ F$ on a suitable finite index subgroup $\Gamma \subset G(R)^+,$ where $F \colon G(R)^+ \to G(B)^+$ is the group homomorphism induced by $f.$ In particular, this confirms a conjecture of Borel and Tits \cite{BT} for Chevalley groups over a field of characteristic zero.
\end{abstract}

\author[I.A.~Rapinchuk]{Igor A. Rapinchuk}

\address{Department of Mathematics, Yale University, New Haven, CT 06502}

\email{igor.rapinchuk@yale.edu}

\maketitle

\section{Introduction and statement of the main results}\label{S:I}

In their fundamental paper \cite{BT}, Borel and Tits showed that if $G$ and $G'$ are algebraic groups defined over infinite fields $k$ and $k'$, respectively, with $G$ absolutely simple, simply-connected, and $k$-isotropic and $G'$ absolutely simple,
then any abstract homomorphism $\rho \colon G(k) \to G'(k')$ between the groups of rational points such that $\rho (G^+)$ is Zariski-dense in $G'(k')$ (here $G^+$ denotes the subgroup of $G(k)$ generated by the $k$-rational points of the unipotent radicals of the parabolic $k$-subgroups of $G$) can (essentially) be written as
a composition $\sigma \circ F$, where $F \colon G(k) \to _{k'}\!G(k')$ is the homomorphism induced by an embedding of fields $f \colon k \to k'$ and $_{k'}G$ is obtained by base change under $f$, and $\sigma \colon _{k'}G \to G'$ is a $k'$-rational morphism of algebraic groups (see \cite{BT}, Theorem A for a more precise statement). We will refer to such a factorization of $\rho$ as a {\it standard description}. A similar, but more technical, statement was also obtained in the case where $G'$ is just assumed to be reductive (see \cite{BT}, 8.16). Later, Seitz \cite{S} established a (generalized form of the) standard description for abstract homomorphisms of universal Chevalley groups assuming that either $k$ is an infinite {\it perfect} field of positive characteristic or $k$ has characteristic zero and $\rho$ maps the elements of $T(k),$ where $T$ is a fixed maximal $k$-torus of $G$, to semisimple elements of $G'(k').$

On the other hand, Borel and Tits pointed out the existence of abstract homomorphisms that fail to have the above description when $G'$ is not necessarily reductive (cf. \cite{BT}, 8.18). The nature of this example prompted the following conjecture (see \cite{BT}, 8.19):
\vskip3mm

\noindent (BT) \ \ \parbox{14cm}{Let $G$ and $G'$ be algebraic groups defined over infinite fields $k$ and $k'$, respectively. If $\rho \colon G(k) \to G'(k')$ is any abstract homomorphism such that $\rho(G^+)$ is Zariski-dense in $G'(k'),$ then {\it there exists a commutative finite-dimensional $k'$-algebra $B$ and a ring homomorphism $f \colon k \to B$ such that  $\rho = \sigma \circ r_{B/k'} \circ F$, where $F \colon G(k) \to _{B}\!G(B)$ is induced by $f$ ($_{B}\!G$ is the group obtained by change of scalars), $r_{B/k'} \colon _{B}\!G(B) \to R_{B/k'} (_{B}\!G)(k')$ is the canonical isomorphism  (here $R_{B/k'}$ denotes the functor of restriction of scalars), and $\sigma$ is a rational $k'$-morphism of $R_{B/k'} (_{B}\!G)$ to $G'.$}}

\vskip3mm

Shortly after the conjecture was formulated, Tits \cite{T} sketched its proof for the case $k = k' = \R.$ The only available result over general fields is due to L.~Lifschitz and A.~Rapinchuk \cite{LR}, where (BT) was essentially proved in the case where $G$ is an absolutely simple simply connected Chevalley group over a field $k$ of characteristic zero and the unipotent radical of $G'$ is commutative. (We note that the unipotent radical of $G'$ in example 8.18 of \cite{BT} was indeed commutative, but in \cite{LR} examples with the unipotent radical having a prescribed nilpotency class were constructed).

On the other hand, there are important results for abstract homomorphisms of higher rank arithmetic groups and lattices. For example, Bass, Milnor, and Serre \cite{BMS} used their solution of the congruence subgroup problem for $G = SL_n (n \geq 3)$ and $Sp_{2n} (n \geq 2)$ to prove that any representation $\rho \colon G(\Z) \to GL_m (\C)$ coincides on a subgroup of finite index $\Gamma \subset G(\Z)$ with the restriction of some rational morphism $\sigma \colon G \to GL_m$ of algebraic groups. Serre \cite{Serre} established a similar result for the group $SL_2 (\Z[1/p]).$ (We note that the results in \cite{BMS} and \cite{Serre} in fact apply to the groups of points over rings of $S$-integers in arbitrary number fields, but their statements in those cases are a bit more technical). Very general results about representations of higher rank arithmetic groups and lattices are also contained in Margulis's Superrigidity Theorem (cf. \cite{Mar}, Chap. VII, 3.10). At the same time, Steinberg \cite{St2} showed that the above results for representations of $SL_n (\Z),$ with $n \geq 3$, can be derived directly from the commutator relations for elementary matrices. Shenfeld \cite{Sh} used the recent result of Kassabov and Nikolov \cite{KN} on the centrality  of the congruence kernel for $\Gamma_{n,k} = SL_n (\Z[x_1, \dots, x_k])$ ($n \geq 3$ and $k \geq 0$), together with a generalization of the techniques of Bass-Milnor-Serre, to answer in the affirmative
the question of D.~Kazhdan on whether every representation $\rho \colon \Gamma_{n,k} \to GL_m (\C)$ with reductive image coincides, on a subgroup of finite index, with a product of specialization maps $\Gamma_{n,k} \to SL_n (\C) \times \cdots \times SL_n (\C)$ followed by a rational map $SL_n \times \cdots \times SL_n \to GL_m$ (his Theorem 1.4 contains this statement in a slightly different, but equivalent, formulation).
The purpose of this paper is to develop Steinberg's generators-relations approach in order to establish a suitable version of the standard description for linear representations of Chevalley groups over general commutative rings, which contains the results in \cite{LR} and \cite{Sh} as very particular cases.

To state our main results, we need to introduce some notations and definitions. Throughout the paper $\Phi$ will denote a reduced irreducible root system of rank $\geq 2$, and $G$ will be the corresponding universal Chevalley-Demazure group scheme over $\Z$ (cf. \cite{Bo1}). For a commutative ring $R$, the pair $(\Phi, R)$ is called {\it nice} if $2 \in R^{\times}$ whenever $\Phi$ contains a subsystem of type $B_2$, and $2,3 \in R^{\times}$ if $\Phi$ is of type $G_2.$ We let $G(R)^+$ denote the subgroup of $G(R)$ generated by the images of the $R$-points of the canonical one-parameter unipotent subgroups corresponding to the roots in $\Phi$; we will refer to $G(R)^+$ as the {\it elementary subgroup} (cf. $\S$ \ref{S:ARR}). Furthermore, let $K$ be an algebraically closed field. For any finite-dimensional $K$-algebra $B$, one can consider the algebraic $K$-group $\mathcal{G} = R_{B/K} (G)$ obtained from $G$ by restriction of scalars (cf. \S \ref{S:R}). Then there is a canonical identification $\mathcal{G} (K) \simeq G(B)$, and in the sequel we will regard $G(B)$ as an algebraic $K$-group via this identification. Our main result is as follows (cf. Theorem \ref{T:R-1}):

\vskip2mm

\noindent {\bf Main Theorem.} \ {\it Let $\Phi$ be a reduced irreducible root system of rank $\geq 2$, $R$ a commutative ring such that $(\Phi, R)$ is a nice pair, and $K$ an algebraically closed field. Assume that $R$ is noetherian if char $K > 0.$ Furthermore let $G$ be the universal Chevalley-Demazure group scheme of type $\Phi$ and let $\rho \colon G(R)^+ \to GL_n (K)$ be a finite-dimensional linear representation over $K$ of the elementary subgroup $G(R)^+ \subset G(R)$. Set $H = \overline{\rho (G(R)^+)}$ (Zariski closure), and let $H^{\circ}$ denote the connected component of the identity of $H$. Then in each of the following situations
\vskip1mm

\noindent {\rm (1)} $H^{\circ}$ is reductive;

\vskip1mm

\noindent {\rm (2)} char $K = 0$ and $R$ is semilocal;

\vskip1mm

\noindent {\rm (3)} char $K = 0$ and the unipotent radical $U$ of $H^{\circ}$ is commutative,

\vskip1mm

\noindent there exists a commutative finite-dimensional $K$-algebra $B$, a ring homomorphism $f \colon R \to B$ with Zariski-dense image, and a morphism $\sigma \colon G(B) \to H$ of algebraic $K$-groups such that for a suitable subgroup $\Gamma \subset G(R)^+$ of finite index, we have
$$
\rho \vert_{\Gamma} = (\sigma \circ F) \vert_{\Gamma},
$$
where $F \colon G(R)^+ \to G(B)^+$ is the group homomorphism induced by $f$.}

\vskip2mm

Thus if $R = k$ is a field of characteristic $\neq 2$ or 3, then
$R$ is automatically semilocal and $(\Phi, R)$ is a nice pair, so
the Main Theorem provides a proof of (BT) in the case that $G(k)$ is split and $k' = K$ is an algebraically closed field of characteristic zero.
In fact, it appears that the techniques used in the proof of the Main Theorem can be generalized to give a standard description for $\rho$ whenever $R$ is any commutative ring such that $(\Phi, R)$ is a nice pair and $K$ is an algebraically closed field of characteristic zero --- the details will be given elsewhere.

One of the central elements in our proof of the Main Theorem is the use of the notion of an algebraic ring. This approach was suggested by Kassabov and Sapir \cite{Kas} in connection with their analysis of the linearity (over fields) of elementary groups over general associative rings. However, the proof of the Main Theorem requires significantly more information about algebraic rings than is given in \cite{Kas}, so $\S$ \ref{S:AR}, which is of independent interest, is devoted to establishing a number of their algebraic and geometric properties. In \S \ref{S:ARR}, we use computations with Steinberg commutator relations to
associate to a given representation $\rho \colon G(R)^+ \to GL_n (K)$ a certain algebraic ring $A$ together with a homomorphism of abstract rings $f \colon R \to A$ having Zariski-dense image. In \S \ref{S:StG}, we lift $\rho$ to a representation $\tilde{\tau} \colon \tilde{G} (A) \to H$ of the Steinberg group $\tilde{G}(A)$ corresponding to the root system $\Phi$,
and also derive some structural information about $\tilde{G}(A).$ Then, in \S \ref{S:FI}, we use the computations of Stein \cite{St2} of the group $K_2$ over semilocal rings to construct an abstract homomorphism $\sigma \colon G(A^{\circ}) \to H$, where $A^{\circ}$ is the connected component of $A$, such that the composition
$$
G(R)^+ \to G(A) \to G(A^{\circ}) \stackrel{\sigma}{\rightarrow} H,
$$
where the first map is induced by $f$ and the second by the natural projection $A \to A^{\circ}$ (cf. Propositions \ref{P:AR-4} and \ref{P:AR-2}), coincides with $\rho$ on a finite index subgroup $\Gamma \subset G(R)^+.$ Finally, in \S \ref{S:R}, we observe that by the results of \S \ref{S:AR}, the connected component $B = A^{\circ}$ is in fact a $K$-algebra in the situations considered in the Main Theorem, and $\sigma$ is a morphism of algebraic groups, which completes the argument. Examples 6.8 and 6.9 discuss how the Main Theorem implies the results of \cite{Sh} and \cite{LR}, respectively.

\vskip2mm

\noindent {\bf Notations and conventions.} Throughout the paper, $\Phi$ will denote a reduced irreducible root system of rank $\geq 2.$ Unless explicitly stated otherwise, all of our rings are commutative and unital. We let $\mathbb{G}_a = {\rm Spec} \ \Z[T]$ and $\mathbb{G}_m = {\rm Spec} \ \Z[T, T^{-1}]$ be the standard additive and multiplicative group schemes over $\Z$, respectively. Also, as noted earlier, if $R$ is a commutative ring, we say that the pair $(\Phi, R)$ is {\it nice} if $2 \in R^{\times}$ whenever $\Phi$ contains a subsystem of type $B_2$, and $2, 3 \in R^{\times}$ if $\Phi$ is of type $G_2.$ Finally, given an algebraic group $H$ (resp., an algebraic ring $A$), we let $H^{\circ}$ (resp., $A^{\circ}$) denote the connected component of the identity (resp., of zero).

\vskip2mm

\noindent {\bf Acknowledgements.} I would like to thank my advisor Professor G.~Margulis for suggesting the topic of this paper and for many helpful discussions. I would also like to thank Professor A.~Lubotzky for bringing to my attention the work of Shenfeld \cite{Sh}. Finally, I would like to thank the referee for detailed suggestions that helped to improve the exposition.

\vskip5mm

\section{On algebraic rings}\label{S:AR}

Our proof of the main theorem relies on some basic results about algebraic rings. In this section, we discuss the notion of algebraic rings, establish some of their basic algebraic and geometric properties, and in some cases completely describe their structure.
All of our algebraic varieties will be over a fixed algebraically closed field $K$.

\vskip1mm

\noindent {\bf Definition 2.1.} An {\it algebraic ring} is a triple $(A, \ba, \bm)$ consisting of an {\it affine} algebraic variety $A$ and two regular maps $\ba \colon A \times A \to A$ and $\bm \colon A \times A \to A$ (``addition" and ``multiplication") such that

\vskip1mm

\noindent (I) \parbox[t]{14.5cm}{$(A, \ba)$ is a commutative algebraic group (in particular, there exists an element $0_A \in A$ such that $\ba (x, 0_A)= \ba (0_A, x) = x$ and there is a regular map $\iota \colon A \to A$ such that $\ba (x, \iota (x)) = \ba (\iota(x), x) = 0_A$, for all $x \in A$);}

\vskip1mm

\noindent (II) \parbox[t]{14.5cm}{$\bm (\bm(x,y), z) = \bm (x, \bm (y,z))$, for all $x, y, z \in A$ (``associativity");}

\vskip1mm

\noindent (III) \parbox[t]{14.5cm}{$\bm (x, \ba (y,z)) = \ba (\bm(x,y), \bm (x, z))$ and $\bm (\ba (x,y), z) = \ba (\bm (x,z), \bm (y,z))$, for all $x, y, z \in A$ (``distributivity").}

\vskip2mm
\noindent An algebraic ring $(A, \ba, \bm)$ is called {\it commutative} if $\bm (x,y) = \bm (y,x)$, for all $x,y \in A$.

\vskip1mm

\noindent The triple $(A, \ba, \bm)$ is an {\it algebraic ring with identity} if in addition to (I)-(III), we have

\vskip1mm

\noindent (IV) \parbox[t]{14.5cm}{there exists an element $1_A \in A$ such that $ \bm  ( 1_A, x) = \bm (x, 1_A) = x$, for all $x \in A.$}

\vskip1mm

\noindent As a matter of convention, all algebraic rings considered in this paper will be assumed to have an identity element.

\vskip2mm

\noindent We will write $x+y$ and $xy$ for $\ba (x,y)$ and $\bm (x,y)$, respectively, whenever this does not lead to confusion.

\vskip1mm

\noindent {\bf Remark 2.2.} Observe that it follows from condition (III) of the definition that for any $a \in A$, the maps
$$
\lambda_a \colon A \to A, \ \ \ x \mapsto \bm  (a,x)
$$
and
$$
\rho_a \colon A \to A, \ \ \ x \mapsto \bm (x,a),
$$
are endomorphisms of the algebraic group $(A, \ba).$ It is also clear from condition (II) that for all $a , b \in A$, we have $\lambda_{ab} = \lambda_a \circ \lambda_b$ and $\rho_{ab} = \rho_{b} \circ \rho_{a}. $

\vskip2mm

\noindent {\bf Definition 2.3.} Let $(A, \ba, \bm)$ and $(A', \ba', \bm')$ be algebraic rings. A regular map $\varphi \colon A \to A'$ is called a {\it homomorphism of algebraic rings} if
$$
\varphi (\ba(x,y)) = \ba' (\varphi (x), \varphi(y)) \ \ \ {\rm and} \ \ \ \varphi (\bm   (x,y)) = \bm' (\varphi(x), \varphi(y)).
$$
If $A$ and $A'$ are rings with identity, we also require that
$$
\varphi (1_A) = 1_{A'}.
$$
If $\varphi$ is in addition an isomorphism of algebraic varieties, then $\varphi$ is called an {\it isomorphism of algebraic rings}.

\vskip2mm

We begin with a description of the group of units $A^{\times}$ of an algebraic ring $A$.

\addtocounter{thm}{3}

\begin{prop}\label{P:AR-1}
Let $A$ be an algebraic ring. Then

\vskip1mm

\noindent {\rm (i)} \parbox[t]{14cm}{The group of units $A^{\times}$ is a principal open subset of $A$.}

\vskip1mm

\noindent {\rm (ii)} \parbox[t]{14cm}{The map $A^{\times} \to A^{\times}, t \mapsto t^{-1}$ is regular. In particular, $(A^{\times}, \bm)$ is an algebraic group.}
\end{prop}
\begin{proof}
(i) Let $\mu^* \colon K[A] \to K[A] \otimes_K K[A]$ be the comorphism corresponding to the multiplication map $\bm \colon A \times A \to A.$ Then, given $f \in K[A]$, we can write $\mu^* (f) = \sum s_i \otimes t_i$ with $s_i, t_i \in K[A].$ This means that
$$
f(xy) = \sum s_i (x) t_i (y)
$$
for all $x,y \in A,$ and in particular, for any $a \in A$,
\begin{equation}\label{E:AR101}
\lambda_a^* (f) = \sum s_i (a) t_i.
\end{equation}
It follows that the functions $\lambda_a^* (f), a \in A$, span a finite-dimensional subspace of $K[A].$ Let $f_1, \dots, f_r$ be a finite system of generators of $K[A]$ as a $K$-algebra. Then the above argument implies that
$$
V : = {\rm span} \{ \lambda_a^* (f_i) \mid a \in A, i = 1, \dots, r \}
$$
is a finite-dimensional subspace of $K[A]$ invariant under all the $\lambda_a^*$.

Let $v_1, \dots, v_n$ be a basis of $V$, and let $\{ w_j \}_{j \in J}$ be a family of elements in $K[A]$ such that $\mathcal{B} = \{ v_1, \dots, v_n \} \cup \{ w_j \}_{j \in J}$ is a $K$-basis of $K[A].$ Let $f \in V.$ Writing the functions $t_i$ in (\ref{E:AR101}) as linear combinations of elements of $\mathcal{B}$, we find that there exist $p_i, q_j \in K[A]$ ($i = 1, \dots, n$, $j \in J$) such that $q_j = 0$ for almost all $j$ and
$$
\lambda_a^* (f) = \sum_{i=1}^n p_i (a) v_i + \sum_{j \in J} q_j (a) w_j
$$
for any $a \in A$. Since $V$ is $\lambda_a^*$-invariant, we see that $q_j = 0$ for all $j$. Thus, there exist $p_{ij} \in K[A]$ such that
\begin{equation}\label{E:AR-50}
\lambda_a^* (v_j) = \sum_{i=1}^n p_{ij} (a) v_i
\end{equation}
for all $a \in A$ and all $j= 1, \dots, n.$ Then $X(a) = (p_{ij} (a))$ is the matrix of the restriction $\ell_a := \lambda_a^* \vert V$ with respect to the basis $v_1, \dots, v_n.$ The function $\chi (a) := \det X(a)$ is regular on $A$, and it is enough to show that $A^{\times}$ coincides with the principal open set
$$
D(\chi) = \{ a \in A \mid \chi (a) \neq 0 \}.$$
For $a \in A^{\times}$, the operator $\lambda_a^*$, hence also $\ell_a$, is invertible, so $\chi (a) \neq 0,$ proving the inclusion $A^{\times} \subset D(\chi).$ Conversely, suppose $\chi (a) \neq 0$. Then $\lambda_a (V) = V,$ and therefore $\lambda_a^* \colon K[A] \to K[A]$ is surjective. Since $A$ is affine, regular functions separate points, so we conclude that $\lambda_a \colon A \to A$ is injective. Let $A^{\circ}$ be the connected component of $0_A$ of the algebraic group $(A, \alpha).$ Since $\lambda_a$ is an endomorphism of the latter, we see that $\lambda_a(A^{\circ})$ is a closed subgroup of $A^{\circ}$ (\cite{Bo}, Corollary 1.4). But the injectivity of $\lambda_a$ implies that $\dim \lambda_a (A^{\circ}) = \dim A^{\circ},$ so actually $\lambda_a (A^{\circ}) = A^{\circ}.$ If $A = \cup_{i=1}^d (a_i + A^{\circ})$ is the decomposition into irreducible components, then
$$
\lambda_a (A) = \cup_{i=1}^d (\lambda_a (a_i) + A^{\circ}).
$$
By the injectivity of $\lambda_a$, the cosets $\lambda_a (a_i) + A^{\circ}$ $(i = 1, \dots, d)$ are all distinct, and therefore $\lambda_a (A) = A.$ In particular, there exists $b \in A$ such that $ab = 1.$ Then $\lambda_a (ba) = a = \lambda_a (1),$ and therefore $ba = 1.$ Thus, $a \in A^{\times}$, as required.

(ii) Let $t \in A^{\times}$. Then it follows from (\ref{E:AR-50}) that for each $j = 1, \dots, n,$ we have
$$
v_j (1) = \sum_{i=1}^n p_{ij} (t) v_i (t^{-1}).
$$
Now since $t$ is a unit, the matrix $X(t) = (p_{ij} (t))$ is invertible, so we can express $v_1 (t^{-1}),$ $\dots,$ $v_n (t^{-1})$ as polynomials in  $p_{ij} (t)$, $i,j = 1, \dots, n$, and $\chi(t)^{-1} = (\det X(t))^{-1}.$ Hence, since $\{v_1, \dots, v_n \}$ generate $K[A]$ and $A^{\times} = D(\chi)$, it follows that the map $A^{\times} \to A^{\times}$, $t \mapsto t^{-1}$ is regular, as claimed.
\end{proof}

Our proof of the main theorem relies on computations, due to M.~Stein \cite{St2}, of the group $K_2$ of a semilocal commutative ring. We will now establish two properties needed for the application of Stein's results, viz.,
that connected algebraic rings are generated by their units (Corollary \ref{C:AR-1}), and that all commutative algebraic rings \footnote{As the referee pointed out, most of the results in this section can be easily generalized to noncommutative algebraic rings. However, since only {\it commutative} algebraic rings arise in the proof of the Main Theorem (in fact, Chevalley groups of type other than $A_n$ cannot be defined over noncommutative rings), we limit our treatment, for the most part, to commutative algebraic rings.} are semilocal (Lemma \ref{L:AR-1}).
\begin{cor}\label{C:AR-1}
Let $A$ be a connected algebraic ring. Then $A = A^{\times} - A^{\times}$.
\end{cor}
\begin{proof}
Since $A$ is connected, it follows from the proposition that $A^{\times}$ is dense in $A$. So, for any $a \in A$, we have $(a + A^{\times}) \cap A^{\times} \neq \varnothing,$ and the required fact follows.
\end{proof}
We also note that Proposition \ref{P:AR-1} puts strong restrictions on algebraic division rings (i.e. algebraic rings $A$ in which $A^{\times} = A \setminus \{ 0 \}$).
\begin{cor}\label{C:AR-3}
Let $A$ be an algebraic division ring. Then $\dim A \leq 1.$
\end{cor}
\begin{proof}
By Proposition \ref{P:AR-1}, $A^{\times} = D(\chi).$ The closed set $V(\chi) := D(\chi)^c$ contains 0, hence is nonempty. By the Dimension Theorem, $\dim V(\chi) \geq \dim A -1.$ On the other hand, $V(\chi) = \{ 0 \},$ so $\dim A \leq 1.$
\end{proof}

%As with any ring, we can regard $A^{\times}$ as an abstract group. In fact, as the following corollary shows, $A^{\times}$ is an algebraic group.

%\begin{cor}\label{C:AR-20}
%The map $A^{\times} \to A^{\times}, t \mapsto t^{-1}$ is regular. In particular, $A^{\times}$ is an algebraic group.
%\end{cor}
%\begin{proof}
%Let $t \in A^{\times}$. Then it follows from (\ref{E:AR-50}) that for each $j = 1, \dots, n,$ we have
%$$
%v_j (1) = \sum_{i=1}^n \alpha_{ij} (t) v_i (t^{-1}).
%$$
%Now since $t$ is a unit, the matrix $X(t) = (\alpha_{ij} (t))$ is invertible, so we can express $v_1 (t^{-1}),$ $\dots,$ $v_n (t^{-1})$ as polynomials in  $\alpha_{ij} (t)$, $i,j = 1, \dots, n$, and $\chi(t)^{-1} = (\det X(t))^{-1}.$ Since $\{v_1, \dots, v_n \}$ is a basis of the vector space $V$, it follows that the coordinates of $t^{-1}$ can also be written as polynomials in $\alpha_{ij} (t)$ and $\chi(t)^{-1}$. So, the fact that $A^{\times} = D(\chi)$ means that
%$t \mapsto t^{-1}$ is a regular map $A^{\times} \to A^{\times}$, as claimed.
%\end{proof}

Algebraic rings arising in the proof of the main result come equipped with a homomorphism of abstract rings $f \colon R \to A$, where $R$ is a given commutative abstract ring, such that $\overline{f(R)} = A$, so we will establish a couple of results (Corollary \ref{C:AR-2} and Lemma \ref{L:AR-3}) relating abstract properties of $R$ and $A$. First, Proposition \ref{P:AR-1} yields the following
\begin{cor}\label{C:AR-2}
Let $f \colon R \to A$ be an abstract ring homomorphism of an abstract commutative semilocal ring into a connected commutative algebraic ring such that $\overline{f(R)} = A.$ Then $\overline{f(R^{\times})} = A.$
\end{cor}
\begin{proof}
Let $\mathfrak{m}_1, \dots, \mathfrak{m}_{\ell}$ be the maximal ideals of $R$ such that $\overline{f(\mathfrak{m}_i)} = A,$ and let $\mathfrak{m}_{\ell+1}, \dots, \mathfrak{m}_n$ be all other maximal ideals. It follows from Proposition \ref{P:AR-1} that $f(\mathfrak{m}_i) \cap A^{\times} \neq \varnothing$ for all $i = 1, \dots, \ell,$ and therefore $f(\mathfrak{m}_1 \cdots \mathfrak{m}_{\ell}) \cap A^{\times} \neq \varnothing.$ The assumption $\overline{f(R)} = A$ implies that $\overline{f (\mathfrak{m}_1 \cdots \mathfrak{m}_{\ell})}$ is an ideal of $A$, so we conclude that $\overline{f(\mathfrak{m}_1 \cdots \mathfrak{m}_{\ell})} = A.$ Now the fact that $A$ is connected implies that $V := A \setminus \cup_{i = \ell + 1}^n \overline{f(\mathfrak{m}_i)}$ is a dense open subset of $A$, and therefore $T:= f(1 + \mathfrak{m}_1 \cdots \mathfrak{m}_{\ell}) \cap V$ is dense in $A$. On the other hand, if $x \in 1 + \mathfrak{m}_1 \cdots \mathfrak{m}_{\ell}$ is such that $f(x) \in V,$ then $x \not\in \cup_{i = 1}^n \mathfrak{m}_i,$ and therefore $x \in R^{\times}.$ Thus, $T \subset f(R^{\times}),$ so the latter is dense in $A$.
\end{proof}

Next, we show that our description of $A^{\times}$ enables us to prove that any commutative algebraic ring $A$ is automatically semilocal as an abstract ring. Before formulating and proving this result, we observe that given an algebraic ring $A$ and a closed 2-sided ideal $\mathfrak{a} \subset A$, the quotient $A/ \mathfrak{a}$ in the category of additive algebraic groups has a natural structure of an algebraic ring. Indeed, since $\mathfrak{a}$ is a closed normal subgroup of the affine group $(A, \ba),$ the quotient $(A/ \mathfrak{a}, \bar{\ba})$ is an affine algebraic group (\cite{Bo}, Theorem 6.8), where $\bar{\ba} \colon A/\mathfrak{a} \times A/\mathfrak{a} \to A/\mathfrak{a}$ is the natural map induced by $\ba$.
Moreover, using the fact that $\mathfrak{a}$ is a 2 -sided ideal of $A,$ it is easy to see that the composite map
$A \times A \stackrel{\bm}{\longrightarrow} A \to A/\mathfrak{a}$ is constant on additive cosets modulo $\mathfrak{a}
\times \mathfrak{a},$ hence factors through $(A \times A)/(\mathfrak{a} \times \mathfrak{a}) \simeq A/\mathfrak{a} \times A/\mathfrak{a}.$
The resulting map $\bar{\bm} \colon A/\mathfrak{a} \times A/\mathfrak{a} \to A/\mathfrak{a}$ and the map $\bar{\ba}$ make $A/\mathfrak{a}$ into an algebraic
ring.
\begin{lemma}\label{L:AR-1}
Let $A$ be a commutative algebraic ring. Then $A$ is semilocal as an abstract ring.
\end{lemma}
\begin{proof}
First, we observe that any abstract maximal ideal $\mathfrak{m} \subset A$ is Zariski-closed. Indeed, if $\mathfrak{m}$ is not closed, then since $\overline{\mathfrak{m}}$ is an ideal, we have $\overline{\mathfrak{m}} = A.$ Then it follows from Proposition \ref{P:AR-1} that $\mathfrak{m} \cap A^{\times} \neq \varnothing$, which is impossible.

Now let $\mathfrak{m}_1, \dots, \mathfrak{m}_n$ be a collection of pairwise distinct maximal ideals of $A$. It follows from the Chinese Remainder Theorem that the canonical map
$$
A \stackrel{\nu}{\rightarrow} A/ \mathfrak{m}_1 \times \cdots A/ \mathfrak{m}_n
$$
is a surjective homomorphism of algebraic rings. If $\dim A/ \mathfrak{m}_i = 0,$ then $A/ \mathfrak{m}_i$ is finite, and consequently $\mathfrak{m}_i \supset A^{\circ}.$ Otherwise, $\dim A/ \mathfrak{m}_i \geq 1.$ The surjectivity of $\nu$ now implies that $n \leq r + \dim A$, where $r$ is the number of maximal ideals of the finite ring $A/ A^{\circ},$ and the required fact follows.
\end{proof}
We have been able to completely describe the structure of algebraic
rings only in characteristic zero (see Proposition \ref{P:AR-2}). However, some important structural information about algebraic rings (particularly
commutative algebraic rings) can be obtained in any characteristic. These results depend on the artinian property, so we address this
issue first.
\begin{lemma}\label{L:AR-2}
Let $A$ be a connected algebraic ring. Then $A$ every right and every left ideal of $A$ is connected and Zariski-closed, hence $A$ is artinian as an abstract ring.
\end{lemma}
\begin{proof}
We will only give the argument for left ideals as the proof for right ideals is completely analogous. Let $\mathfrak{a} \subset A$ be a left ideal.
For any $a_1, \dots, a_n \in \mathfrak{a}$, the left ideal $\mathfrak{b} \subset A$ generated by $a_1, \dots, a_n$ is the image of the homomorphism of algebraic groups $A^n \to A$, $(x_1, \dots, x_n) \mapsto x_1 a_1 + \dots + x_n a_n,$ and therefore is closed (\cite{Bo}, Corollary 1.4) and connected. Pick $a_1, \dots, a_n \in \mathfrak{a}$ so that the corresponding $\mathfrak{b}$ has maximum possible dimension. Then $\mathfrak{a} = \mathfrak{b}.$ Indeed, let $a \in \mathfrak{a}$ and $\mathfrak{b}' = Aa + \mathfrak{b}.$ Then $\mathfrak{b}'$ is a closed connected left ideal contained in $\mathfrak{a}$ and $\mathfrak{b} \subset \mathfrak{b}'.$ By dimension considerations, we conclude that $\mathfrak{b} = \mathfrak{b}'$, so $a \in \mathfrak{b}.$ Thus, $\mathfrak{a} = \mathfrak{b},$ hence closed.
Since $A$ is a noetherian topological space for the Zariski topology (i.e. closed sets satisfy the descending chain condition), it now follows that $A$ is artinian.
\end{proof}

\vskip1mm

\noindent {\bf Remark 2.10.} We note that without the assumption of connectedness, the conclusion of the lemma may fail. For example, suppose $K$ is an algebraically closed field of characteristic $p > 0.$ Let $A_0 = K \times K,$ with the usual addition and multiplication given by
$$
\bm((x_1, y_1), (x_2, y_2)) = (x_1 y_2 + x_2 y_1, y_1 y_2).
$$
It is easily seen that $A_0$ is a commutative algebraic ring with identity element $(0,1).$ Then $A = K \times \mathbb{F}_p,$ where $\mathbb{F}_p$ is the prime subfield of $K$, is an algebraic subring of $A_0.$ Now if $S \subset K$ is any additive subgroup, then $\mathfrak{a} = (S,0)$ is an abstract ideal of $A$. Hence $A$ is not artinian and not every ideal of $A$ is Zariski-closed.

%the structure theorem for commutative artinian rings (cf. \cite{At}, Theorem 8.7), we have
%$$
%A = A_1 \times \cdots A_n,
%$$
%where each $A_i$ is an algebraic ring and local and artinian as an abstract ring.
%It follows that if $\mathfrak{R}(A)$ is the Jacobson radical of

\vskip2mm

\addtocounter{thm}{1}

Nevertheless, the artinian property does hold for commutative algebraic rings satisfying one additional condition, which we will now define. Let $A$ be a commutative algebraic ring. The connected component $A^{\circ}$ of $0_A$ in $(A, \ba)$ is easily seen to be an ideal of $A$. Consider the following condition on $A$:
\vskip3mm

\noindent (FG)\ \  \parbox[t]{15cm}{$A^{\circ}$ is finitely generated as an ideal of $A$.}

\vskip3mm

It turns out that commutative algebraic rings satisfying (FG) possess a number of important structural properties.
\begin{prop}\label{P:AR-4}
Let $A$ be a commutative algebraic ring satisfying {\rm (FG)}. Then
\vskip1mm
\noindent {\rm (i)} \parbox[t]{15cm}{Every abstract ideal ideal $\mathfrak{a} \subset A$ is Zariski-closed, and consequently $A$ is artinian.}
\vskip1mm
\noindent {\rm (ii)} \parbox[t]{15cm}{We have the following direct sum decomposition of algebraic rings
$$
A = A^{\circ} \oplus C,$$
where $C$ is a finite ring isomorphic to $A/ A^{\circ}$.}
\end{prop}
The key step in the proof of the proposition is the following
\begin{lemma}\label{L:AR-4}
Let $A$ be a commutative algebraic algebraic ring satisfying {\rm (FG)}. Then $A^{\circ}$ is an algebraic ring with identity.
\end{lemma}
\begin{proof}
Suppose $A^{\circ} = Aa_1 + \cdots + A a_r.$ Write
$$
A = \bigcup_{i=1}^s (c_i + A^{\circ}).
$$
Then
\begin{equation}\label{E:AR102}
A^{\circ} = \bigcup_{i_1, \dots, i_r \in \{1, \dots, s \}} (c_{i_1} a_1 + \cdots + c_{i_r} a_r + \mathfrak{a}),
\end{equation}
where $\mathfrak{a} = A^{\circ} a_1 + \cdots + A^{\circ} a_r.$ Notice that being the image of the group homomorphism $A^{\circ} \times \cdots \times A^{\circ} \to A^{\circ},$ $(x_1, \dots, x_r) \mapsto x_1 a_1 + \cdots + x_r a_r$, of additive groups, $\mathfrak{a}$ is Zariski-closed.
%as it is the image of the group homomorphism
%$A^{\circ} \times A^{\circ} \to A^{\circ},$ $(x_1, \dots, x_r) \mapsto x_1 a_1 + \cdots + x_r a_r$.
Since $A^{\circ}$ is irreducible, we conclude from (\ref{E:AR102}) that $\mathfrak{a} = A^{\circ}.$ This means that for each $i = 1 , \dots, r,$ we have relations
$$
a_i = s_{i1} a_1 + \cdots + s_{ir} a_r,
$$
with $s \in A^{\circ}.$ Consider the matrix $M = (\delta_{ij} - s_{ij}).$ Then
$$
M \left( \begin{array}{c} a_1 \\ \vdots \\ a_r \end{array} \right) = \left( \begin{array}{c} 0 \\ \vdots \\ 0 \end{array} \right).
$$
Multiplying the latter on the left by the classical adjoint of $M$ shows that $d = \det (\delta_{ij} - s_{ij})$, computed in $A$, annihilates every $a_i,$ hence all of $A^{\circ}.$ On the other hand, $d = 1 - d_0$ for some $d_0 \in A^{\circ}.$ Then for any $x \in A^{\circ},$ we have $d_0 x = (1-d)x = x,$ i.e. $d_0$ is an identity for $A^{\circ}.$
\end{proof}

\noindent {\it Proof of Proposition \ref{P:AR-4}}: (i) By the lemma, $A^{\circ}$ is a connected algebraic ring with identity, so by Lemma \ref{L:AR-2}, every abstract ideal of $A^{\circ}$ is Zariski-closed. If $\mathfrak{a} \subset A$ is an abstract ideal, then $\mathfrak{a}_0 := \mathfrak{a} \cap A^{\circ}$ is closed. But $[\mathfrak{a} : \mathfrak{a}_0] < \infty,$ so $\mathfrak{a}$ is also closed.

\noindent (ii) By Lemma \ref{L:AR-4}, there is an identity element $e \in A^{\circ}.$
It is clear that $e$ is idempotent and that $A^{\circ} = eA$ (since $A^{\circ}$ is an ideal).
Therefore, we have the following direct sum decomposition
$$
A = A^{\circ} \oplus C, \ \ \ a \mapsto (ea, (1-e)a),
$$
where $C = (1-e)A$. Since $[A : A^{\circ}] < \infty,$ $C$ is a finite ring.

%\begin{cor}\label{C:AR-4}
%Let $A$ be a commutative algebraic ring satisfying {\rm (FG)}. Then every abstract ideal ideal $\mathfrak{a} \subset A$ is Zariski-closed, and consequently $A$ is %artinian.
%\end{cor}
%\begin{proof}
%Since $A^{\circ}$ is a connected algebraic ring by the preceding lemma, it follows from the proof of Lemma \ref{L:AR-2} that every abstract ideal of $A^{\circ}$ is Zariski-closed. If $\mathfrak{a} \subset A$ is an abstract ideal, then $\mathfrak{a_0} := \mathfrak{a} \cap A^0$ is closed. But $[\mathfrak{a} : \mathfrak{a}_0] < \infty,$ so $\mathfrak{a}$ is also closed.
%\end{proof}
%\begin{cor}\label{C:AR-5}
%Suppose $A$ is a commutative algebraic ring satisfying {\rm (FG)}. Then $A = A^{\circ} \oplus C,$ where $C$ is a finite ring.
%\end{cor}
%\begin{proof}
%By Lemma \ref{L:AR-4}, there is an identity element $e \in A^{\circ}.$ Now $e$ is idempotent and $A^{\circ} = e A$ as $A^{\circ}$ is an ideal. Therefore, we have the following direct sum decomposition
%$$
%A = A^{\circ} \oplus e' A,
%$$
%where $e' = 1-e.$ Since $[A : A^{\circ}] < \infty,$ $C := e' A$ is a finite ring.
%\end{proof}
An important class of examples of algebraic rings satisfying condition (FG) is obtained as follows.
\begin{lemma}\label{L:AR-3}
Let $f \colon R \to A$ be an abstract homomorphism of an abstract commutative ring $R$ into a commutative algebraic ring $A$ such that $\overline{f(R)} = A.$
\vskip1mm

\noindent {\rm (i)} \parbox[t]{15cm}{If $R$ is noetherian, then $A$ satisfies {\rm (FG)}.}

\vskip1mm

\noindent {\rm (ii)} \parbox[t]{15cm}{If $R$ is an infinite field, then $A$ is connected}

\end{lemma}
\begin{proof}
(i) Since $A^{\circ}$ is open, for $\mathfrak{r} = f^{-1} (f(R) \cap A^{\circ})$, we have $\overline{f(\mathfrak{r})} = A^{\circ}.$ By assumption, $R$ is noetherian and clearly $\mathfrak{r}$ is an ideal, so we have $\mathfrak{r} = R u_1 + \dots + R u_m.$ Then
since $A f(u_1) + \cdots + A f(u_m)$ is closed and contains $f(\mathfrak{r}),$ we obtain $A^{\circ} = A f(u_1) + \cdots + A f(u_m).$

\vskip1mm
\noindent (ii) If $A \neq A^{\circ},$ then $f^{-1}(A^{\circ})$ would be a proper ideal of $R$ of finite index, which is impossible.
\end{proof}

Suppose $S$ is a finite-dimensional $K$-algebra. Then $S$ has a natural structure of a connected algebraic ring. We say that an algebraic ring $A$ {\it comes from an algebra} if there exists a finite dimensional $K$-algebra $S$ and an isomorphism $A \simeq S$ of algebraic rings. Furthermore, we say that an algebraic ring $A$ {\it virtually comes from an algebra} if $A \simeq A' \oplus B$, where
$A'$ comes from an algebra and $B$ is finite. For such algebraic rings most of our previous results are immediate.
On the other hand, it turns out that in characteristic zero, all algebraic rings virtually come from algebras.
\begin{prop}\label{P:AR-2}
Let $A$ be an algebraic ring over an algebraically closed field $K$ of characteristic zero. Then $A$ virtually comes from an algebra.
\end{prop}
The proof relies on the following lemma which is true in any characteristic.
\begin{lemma}\label{L:AR-5}
Suppose $A$ is an algebraic ring over an algebraically closed field $K$ of any characteristic. Then $A = A' \oplus C,$ where $A'$ is an algebraic subring of $A$ consisting of all unipotent elements in $(A, \ba)$, and $C$ is a finite ring consisting of all semisimple elements. In particular, if $A$ is connected, then $A$ consists entirely of unipotent elements.
\end{lemma}
\begin{proof}

Let $A'$ and $C$ denote the sets of unipotent and semisimple elements in $(A, \ba),$ respectively. By (\cite{Bo}, Theorem 4.7), $A'$ and $C$ are closed subgroups of $(A, \ba)$ and
\begin{equation}\label{E:AR103}
(A,  \ba) \simeq (A', \ba \vert A') \oplus (C, \ba \vert C)
\end{equation}
as algebraic groups.
As we pointed out in Remark 2.2, for a fixed $a \in A$, the maps $\lambda_a \colon x \mapsto ax$ and $\rho_a \colon x \mapsto xa$ are endomorphism of $(A,  \ba),$ so it follows from (\cite{Bo}, Theorem 4.4(iv)), that $aA' , A'a \subset A'$ and $a C, C a \subset C$. Thus, $A'$ and $C$ are 2-sided ideals of $A$ and therefore the decomposition in (\ref{E:AR103}) is a direct sum of rings with identity. Let us show that $C$ is finite.
Consider the map $\bm_C \colon C^{\circ} \times C \to C$ induced by the multiplication map $\bm$. Applying
(\cite{Bo}, Proposition 8.10(iii)), we conclude that $\bm (s,t)$ is independent of $s.$ Plugging in $s = 0$, we obtain that $C^{\circ} C = \{ 0 \}.$ On the other hand, if $1_{C} \in C$ is the identity in $C$ (the projection of $1_A$ to $C$), then $s \cdot 1_{C} = s$ for any $s \in C$. So $C^{\circ} C \supset C^{\circ}.$ Thus, $C^{\circ} = \{0 \}$, and therefore $C$ is finite.
\end{proof}

\noindent {\it Proof of Proposition \ref{P:AR-2}}. In view of the lemma, it remains to show that $A'$ comes from a $K$-algebra.
Since char $K = 0,$ it follows from (\cite{Bo}, Remark 7.3), that there exists an isomorphism  $\sigma \colon (A',$ $\ba \vert A') \to (K^n, +)$ of additive algebraic groups, where $n = \dim A.$ It suffices to show that $\bm' \colon K^n \times K^n \to K^n,$ $\bm' = \sigma \circ \bm \vert_{A'} \circ \sigma^{-1}$, is $K$-bilinear.
It follows from condition (III) in Definition 2.1 that
$$
\bm'(m x_1, x_2) = m  \bm' (x_1, x_2) \ \ \ {\rm and} \ \ \ \bm'(x_1, m x_2) = m  \bm (x_1, x_2)
$$
for all $x_1, x_2 \in K^n$ and $m \in \Z.$ Since $\Z \subset K$ is infinite and $\bm'$ is regular, we conclude that these equalities hold for all $m \in K,$ verifying that $\bm'$ is bilinear.

\vskip2mm

\noindent {\bf Remark 2.16.} (1) In \cite{Kas}, the authors give a proof of Proposition \ref{P:AR-2} for $K = \C$ using a topological argument.

(2) It follows from Proposition \ref{P:AR-2} that condition (FG) always holds if char $K = 0.$ More generally, for $K$ of any characteristic, condition (FG) is equivalent to the fact that $A^{\circ}$ coincides with $A^{\circ} A^{\circ} := \{ \sum a_i b_i \mid a_i, b_i \in A^{\circ} \}$. Indeed, one implication follows from Lemma \ref{L:AR-4}, while to prove the converse, one needs to mimic the argument of Lemma \ref{L:AR-2}.

\vskip2mm

As the following simple example shows, the assertion of Proposition \ref{P:AR-2} may fail in positive characteristic.

\vskip1mm

\noindent {\bf Example 2.17.} Suppose $K$ is an algebraically closed field of characteristic $p > 0.$ Let $\tilde{A} = K \times K$, with the usual addition, and multiplication given by
$$
\bm ((x_1, y_1), (x_2, y_2)) = (x_1 x_2, x_1^p y_2 + x_2^p y_1).
$$
It is clear that $\tilde{A}$ is an algebraic ring with identity element $(1,0).$ However, $\tilde{A}$ is not a $K$-algebra. For example, the map $K \times \{0 \} \to \{ 0 \} \times K,$ $(x,0) \mapsto \bm ((x, 0), (0,1)) = (0,x^p)$ is bijective but not an isomorphism of algebraic rings. However, such a bijection would be an isomorphism in a $K$-algebra.

The algebraic ring $\tilde{A}$ in this example is related to an algebra as follows: let $A = K \times K$, with the usual addition, and multiplication given by
$$
\bm ((x_1, y_1), (x_2, y_2)) = (x_1 x_2, x_1 y_2 + x_2 y_1).
$$
Then the map $(x, y) \mapsto x + \delta y$ identifies $A$ with the algebra $K[\delta],$ $\delta^2 = 0$, of dual numbers. On the other hand, the map $\tilde{A} \to A$, $(x,y) \mapsto (x^p, y),$ is a homomorphism of algebraic rings, which is an isomorphism of abstract rings, but not an isomorphism of algebraic rings.

As the referee pointed out, the truncated Witt vectors with coefficients in a field $K$ of characteristic $p > 0$ (cf. \cite{Serre1}, Chapter II, \S 6) provide a series of similar examples of algebraic rings that are not algebras. In this connection, we would like to formulate the following conjecture that would enable one to extend some aspects of the Main Theorem (in particular, part (2)) to the case of $K$ of positive characteristic.

\vskip2mm

\noindent {\bf Conjecture 2.18.} {\it Let $A$ be a connected algebraic ring over an algebraically closed field $K$ of characteristic $p >0$ such that $p A = 0$.\footnote{See (\cite{Kas}, Remark 6) for an example of an algebraic ring where the exponent of the additive group is not equal to the characteristic of the field. Thus, this assumption cannot be omitted.} Then there exists a finite-dimensional $K$-algebra $A'$ and a surjective (maybe even bijective) homomorphism of algebraic rings $\rho \colon A' \to A$.}

%Nevertheless, the example does suggest the following conjecture.

\vskip3mm

%\noindent {\bf Conjecture 2.18:} Suppose $A$ is a connected algebraic ring over an %algebraically closed field $K$ of characteristic $p>0$, and assume that the multiplication in $A$ is given by functions of degree $< p,$
%i.e. there exists a system of generators $t_1, \dots, t_n$ of $K[A]$ such that for the %comorphism $\mu^* \colon K[A] \to K[A] \otimes_K K[A],$ we have
%$$
%\mu^* (t_i) = \sum \varphi_{ij} \otimes \psi_{ij},
%$$
%where $\varphi_{ij}, \psi_{ij}$ are polynomials of degree $< p$ in $t_1, \dots, t_n.$ Then $A$ comes from a $K$-algebra.

%\vskip2mm

We will  now establish some structural results for commutative artinian algebraic rings that hold regardless of whether
or not the ring (virtually) comes from an algebra. So, let $A$ be a commutative artinian algebraic ring with identity,
${\mathfrak m}_1, \ldots , {\mathfrak m}_r$ be its maximal ideals (note that $A$ is semilocal by \cite{At}, Proposition 8.3), and $J = {\mathfrak m}_1 \cap \cdots \cap {\mathfrak m}_r$
be its Jacobson radical (recall that the ${\mathfrak m}_i$'s, hence also $J,$ are Zariski-closed by Lemma \ref{L:AR-1}). It is well-known (cf. \cite{At}, Theorem 8.7) that
there are idempotents $e_1, \ldots , e_r \in A$ such that $e_1 + \cdots + e_r = 1_A,$ $e_ie_j = 0$ for $i \neq j,$ and $A_i := e_iA$ is a {\it
local} commutative artinian ring with identity for each $i = 1, \ldots , r.$ We have
\begin{equation}\label{E:AR120}
A \simeq A_1 \oplus \cdots \oplus A_r, \  \  \  a \mapsto (e_1a, \ldots , e_ra),
\end{equation}
as algebraic rings. Furthermore, after possible renumbering, the ideal ${\mathfrak m}_i$ corresponds to $A_1 \oplus \cdots \oplus {\mathfrak m}'_i
\oplus \cdots A_r$, where ${\mathfrak m}'_i$ is the unique maximal ideal of $A_i.$

Let now $B$ be a local commutative artinian algebraic ring with maximal ideal ${\mathfrak n}.$ It follows
from Corollary \ref{C:AR-3} that $\dim B/{\mathfrak n} \leq 1.$ If $\dim B/{\mathfrak n} = 0,$ i.e. $B/{\mathfrak n}$ is finite, then since
there exists $n \geq 1$ such that ${\mathfrak n}^n = \{ 0 \}$ (\cite{At}, Proposition 8.4) and each quotient ${\mathfrak n}^j/{\mathfrak n}^{j+1}$ in the
filtration
\begin{equation}\label{E:AR121}
B \supset {\mathfrak n} \supset {\mathfrak n}^2 \supset \cdots \supset {\mathfrak n}^{n-1} \supset {\mathfrak n}^n = \{ 0 \}
\end{equation}
is a finitely generated $B/{\mathfrak n}$-module (as it is an artinian module over the field $B/ \mathfrak{n}$), we conclude that $B$ itself is finite. Now, suppose $\dim B/{\mathfrak n} = 1.$
Since $B/{\mathfrak n}$ is an infinite algebraic division ring, it is automatically connected, so we can use the following.

\addtocounter{thm}{3}

%We conclude this section with two results on the geometry of algebraic rings. As before, we again assume that $K$ is any algebraically closed field. It is well-known that a connected one-dimensional affine algebraic group is isomorphic to either $\mathbb{G}_a$ or $\mathbb{G}_m.$ It turns out that one-dimensional connected commutative algebraic rings are isomorphic to $K$. The precise statement is as follows.
\begin{prop}\label{P:AR-3}
Suppose $(A, \ba, \bm)$ is a one-dimensional connected commutative algebraic ring. Then $(A, \ba , \bm) \simeq (K, +, \cdot)$ as algebraic rings.
\end{prop}
\begin{proof}
By Lemma \ref{L:AR-5}, $(A, \ba)$ is unipotent. Hence $(A, \ba) \simeq \mathbb{G}_a$ by (\cite{Bo}, Theorem 10.9). To complete the proof, we need to describe the multiplication map $\bm.$ In any case, $\bm$ is given by a polynomial $\bm (x,y).$
Let $m = {\rm deg}_x \bm (x,y).$ Then by condition (II) in Definition 2.1, we have
$$
m^2 = {\rm deg}_x  \bm (\bm(x,y), z) = {\rm deg}_x \bm (x, \bm(y,z)) = m.
$$
So, $m$ is 0 or 1. If $m=0,$ then since $\bm(x,y) = \bm  (y,x),$ ${\rm deg}_y  \bm (x,y) = 0$, and therefore $\bm$ is constant. Hence $\bm \equiv 0,$ and there is no identity, a contradiction. So, $m = 1.$ Then also ${\rm deg}_y  \bm (x,y) = 1,$ and therefore
$$
\bm (x,y) = axy + b(x+ y) + c.
$$
The identities $\bm (x,0) \equiv 0$ and $\bm (0,y) \equiv 0$ force $b = c = 0.$ Thus $\bm  (x,y) = axy.$ Consider $\sigma \colon K \to K$, $\sigma (x) = a x.$ Then $\sigma (\bm (x, y)) = a^2 xy = (ax)(ay).$ In other words, if we let $\bm' \colon K \times K \to K,$ $\bm' (x,y) = xy,$ then the diagram
%$$
%\begin{array}{ccc} K \times K & \stackrel{\sigma \times \sigma}{\longrightarrow} & K \times K \\ \bm \downarrow & & \downarrow \bm' \\ K & \stackrel{\sigma}{\longrightarrow} & K \end{array}
%$$
$$
\xymatrix{K \times K \ar[d]_{\bm} \ar[r]^{\sigma \times \sigma} & K \times K \ar[d]^{\bm'} \\ K \ar[r]^{\sigma} & K}
$$
commutes. Thus, $\sigma$ gives an isomorphism of algebraic rings $(K, \ba, \bm) \simeq (K, +, \cdot).$
\end{proof}

So, if $\dim  B/{\mathfrak n} = 1$, then $B/{\mathfrak n} \simeq K$ with the natural operations. Combining this with the decomposition (\ref{E:AR120}) and taking into account that
$J = {\mathfrak m}'_1 \oplus \cdots \oplus {\mathfrak m}'_r,$ we obtain the following
\begin{prop}\label{P:AR-20}
Let $A$ be a commutative artinian algebraic ring with Jacobson radical $J.$ Then

\vskip1mm

\noindent {\rm (i)} $A \simeq A_1 \oplus \cdots \oplus A_r$, where each $A_i$ is a local commutative artinian
algebraic ring;

\vskip1mm

\noindent {\rm (ii)} \parbox[t]{15cm}{$A/{J} \simeq (K, +, \cdot)^n \oplus C$ where, $n = \dim A/J$ and $C$ is
a finite algebraic ring; in particular, $A/J$ always virtually comes from an algebra.}

\end{prop}

Finally, we would like to explain the requirement in Definition 2.1 that the algebraic variety underlying an algebraic
ring be {\it affine.} Of course, any commutative algebraic group $(A , \ba)$ (in particular, any abelian variety) can be made
into an algebraic ring satisfying conditions (I)-(III) of Definitions 2.1 by taking $\bm \colon A \times A \to A,$ $\bm (x , y) = 0_A$ for
all $x , y \in A.$ This ring, however, will not satisfy condition (IV) of the definition (unless it is trivial). We will now show that only affine varieties can
support the structure of an algebraic ring satisfying all of the conditions (I)-(IV).
\begin{thm}\label{T:AR-1}
Let $A$ be an irreducible algebraic variety equipped with regular maps $\ba \colon A \times A \to A$ and $\bm \colon A \times A \to A$ satisfying
conditions {\rm (I)-(IV)} of Definition 2.1 (in other words, giving $A$ the structure of an algebraic ring with identity). Then $A$ is affine.
\end{thm}

The first step in the proof is to consider the case where $A$ is complete.
\begin{lemma}\label{L:AR-6}
Let $A$ be a \emph{complete} irreducible variety equipped with regular morphisms  $\ba \colon A \times A \to A$ and
$\bm \colon A \times A \to A$ satisfying conditions {\rm (I)-(III)} of Definition 2.1. Then $\bm \equiv 0_A.$
\end{lemma}

The proof relies on the following well-known result from (\cite{Mum}, Chapter 2).
\begin{lemma}\label{L:AR-7}
{\rm (Rigidity Lemma, Form I)} Let $X$ be a complete irreducible variety, $Y$ and $Z$ any irreducible varieties, and $f \colon X \times Y \to Z$ a morphism such that for some $y_0 \in Y,$ $f(X \times \{y_0 \})$ is a single point $z_0$ of $Z$. Then there is a morphism $g \colon Y \to Z$ such that if $p_2 \colon X \times Y \to Y$ is the projection, $f = g \circ p_2.$
\end{lemma}

Now, to prove Lemma \ref{L:AR-6}, suppose $A$ is a complete irreducible variety with the structure of an algebraic ring. Since $\bm (A, 0_A) = \{ 0_A \}$, the Rigidity Lemma gives a morphism $\nu \colon A \to A$ such that $\bm = \nu \circ p_2.$ Thus $\bm (a,b) = \nu (b)$ for all $a, b \in A$, i.e. $\bm$ is independent of $a.$ Setting $a = 0_A,$ we obtain that $\bm \equiv 0_A.$

\vskip1mm

\noindent {\it Proof of Theorem \ref{T:AR-1}}. We need to recall the following celebrated result of Chevalley \cite{C} (see \cite{Con} for a modern proof).
\begin{thm}
Let $K$ be an algebraically closed field and $G$ a connected algebraic group over $K.$ Then there exists a unique connected normal affine closed subgroup $H$ of $G$ for which $G/H$ is an abelian variety.
\end{thm}

\vskip1mm

Now let $A$ be an irreducible variety equipped with regular maps
$\ba$ and $\bm$ satisfying conditions (I)-(IV) of Definition 2.1. By Chevalley's Theorem,
there exists a closed connected affine subgroup $\mathfrak{a}$ of  $(A, \ba)$ such that $A/ \mathfrak{a}$ is an abelian variety (in particular, it is complete).
Let us show that $\mathfrak{a}$ is in fact a 2-sided ideal of $A$. For this, we need to prove that $a \mathfrak{a} \subset \mathfrak{a}$ and $\mathfrak{a} a \subset \mathfrak{a}$ for any $a \in A$. We only give a proof of the first inclusion as the second one is proved completely analogously.

By Remark 2.2, the map $\lambda_a \colon A \to A$ is a morphism of algebraic groups, so $a \mathfrak{a}$ is a closed subgroup of $A$. Moreover, it is affine.
Indeed, let $\mathfrak{b} = \{x \in \mathfrak{a} \mid ax = 0 \}.$ This is clearly a closed normal subgroup of $\mathfrak{a},$ so the quotient $\mathfrak{a}/ \mathfrak{b}$ is an affine algebraic group (\cite{Bo}, Theorem 6.8). On the other hand, $\lambda_a$ induces a bijective morphism $\mathfrak{a}/ \mathfrak{b} \to a \mathfrak{a}$, and therefore $a \mathfrak{a}$ is affine by (\cite{Bo}, Proposition AG 18.3).
Hence $\mathfrak{a} \oplus a \mathfrak{a}$ is a connected affine algebraic group.

Now let $\gamma \colon \mathfrak{a} \oplus a \mathfrak{a} \to \mathfrak{a} + a \mathfrak{a}$, $(x,y) \mapsto x + y$. By condition (I) of Definition 2.1, this is a morphism of algebraic groups, so $\mathfrak{a} + a \mathfrak{a}$ is a closed subgroup of $A.$ We will now show that it is affine. Let $\mathfrak{c} = \mathfrak{a} \cap a \mathfrak{a}$, embedded into $\mathfrak{a} \oplus a \mathfrak{a}$ via $x \mapsto (x, -x).$ Clearly $\mathfrak{c}$ is a closed normal subgroup of $\mathfrak{a} \oplus a \mathfrak{a},$ so $(\mathfrak{a} \oplus a \mathfrak{a})/ \mathfrak{c}$ is affine. But $\gamma$ induces a bijection $(\mathfrak{a} \oplus a \mathfrak{a})/ \mathfrak{c} \to \mathfrak{a} + a \mathfrak{a},$ so as above, we conclude that $\mathfrak{a} + a \mathfrak{a}$ is affine. Thus, $(\mathfrak{a} + a \mathfrak{a})/ \mathfrak{a}$ is a closed, affine connected subgroup of the complete variety $A/ \mathfrak{a}$ (\cite{Bo}, Corollary 6.9). Therefore, it consists of just a single point. In other words, $a \mathfrak{a} \subset \mathfrak{a}$, as needed.

Thus, $\mathfrak{a}$ is a 2-sided ideal of $A$, so as remarked before the statement of Lemma \ref{L:AR-1}, we obtain an algebraic ring with identity $(A/ \mathfrak{a}, \bar{\ba}, \bar{\bm}).$
But $A/ \mathfrak{a}$ is complete, so $\bar{\bm} \equiv 0_{\bar{A}}$ by Lemma \ref{L:AR-6}. Since $A/ \mathfrak{a}$ contains an identity element, it follows that $A = \mathfrak{a}$, i.e. $A$ is affine.

\vskip5mm

\section{An algebraic ring associated to a representation of $G(R)^+$}\label{S:ARR}

Let $\Phi$ be a reduced irreducible root system of rank $\geq 2$, and let
$G$ be the corresponding universal Chevalley group scheme over $\Z$ (cf. \cite{Bo1}). Then, in particular, for every root $\alpha \in \Phi$, we have a canonical morphism of group schemes $e_{\alpha} \colon \mathbb{G}_a \to G$, where $\mathbb{G}_a = {\rm Spec} \  \Z[T]$ is the standard additive group scheme over $\Z.$ For any commutative ring $R$, the group of $R$-points $G(R)$ is the {\it universal Chevalley group of type $\Phi$ over $R$}. Furthermore, for $\alpha \in \Phi$, the morphism $e_{\alpha}$ induces a group homomorphism $R^+ \to G(R)$, which will also be denoted $e_{\alpha}$ (rather than $(e_{\alpha})_R$) whenever this does not cause confusion. Then $e_{\alpha}$ is an isomorphism between $R^+$ and the subgroup $U_{\alpha} (R) := e_{\alpha} (R)$ of $G(R).$ The subgroup of $G(R)$ generated by the $U_{\alpha} (R)$, for all $\alpha \in \Phi$, will be denoted by $G(R)^+$ and called the {\it elementary subgroup of $G(R)$}.

Now let $K$ be an algebraically closed field (of any characteristic), and let $\rho \colon G(R)^+ \to GL_n (K)$ be a finite-dimensional (abstract) representation. The goal of this section is to associate to $\rho$ an algebraic ring $A$ together with a homomorphism of abstract rings
$f \colon R \to A$ with Zariski-dense image that satisfy some natural properties. This construction, which relies on explicit computations with Steinberg relations in $G(R)^+$, was described in \cite{Kas} for groups of type $A_2.$ The following theorem generalizes it to groups of all types, with some minor restrictions on $R$ in the cases where $\Phi$ contains roots of different lengths.

\begin{thm}\label{T:ARR-1}
Let $\Phi$ be a reduced irreducible root system of rank $\geq 2$,  $G$ the corresponding universal Chevalley group scheme, and $R$ a commutative ring such that $(\Phi, R)$ is a nice pair \footnotemark. Then, given a representation $\rho \colon G(R)^+ \to GL_n (K)$, there exists a commutative algebraic ring $A$ with identity together with a homomorphism of abstract rings $f \colon R \to A$ having Zariski-dense image such that for every root $\alpha \in \Phi$, there is an injective regular map $\psi_{\alpha} \colon A \to H$ into $H := \overline{\rho (G(R)^+)}$ (Zariski closure) satisfying
\begin{equation}\label{E:ARR101}
\rho (e_{\alpha} (t)) = \psi_{\alpha} (f(t)).
\end{equation}
for all $t \in R.$ \footnotetext{Recall that this means that $2 \in R^{\times}$ whenever $\Phi$ contains a subsystem of type $B_2$ and $2, 3 \in R^{\times}$ if $\Phi$ is of type $G_2$.}
\end{thm}

We begin with some simple reductions. First, suppose we have been able to construct an algebraic ring $A$ and a ring homomorphism $f \colon R \to A$ with Zariski-dense image such that a regular map $\psi_{\alpha_0} \colon A \to H$ satisfying (\ref{E:ARR101}) exists for {\it some} root $\alpha_0 \in \Phi.$ Then regular maps $\psi_{\alpha} \colon A \to H$ satisfying (\ref{E:ARR101}) exist for {\it all} roots $\alpha \in \Phi$ having the same length as $\alpha_0 .$ Indeed, it is well-known (e.g., see \cite{H1}, 10.4, Lemma C) that the Weyl group $W(\Phi)$ of $\Phi$ acts transitively on roots of each length. So, it follows from \cite{St1}, 3.8, relation (R4), that there exists $w \in G(R)^+$ such that
\begin{equation}\label{E:ARR102}
w e_{\alpha_0}(t) w^{-1} = e_{\alpha} (\varepsilon(w) t)
\end{equation}
for all $t \in R$, where $\varepsilon(w) \in \{\pm 1 \}$ is independent of $t$.  Now define $\psi_{\alpha} \colon A \to H$ by
$$
\psi_{\alpha} (a) = \rho (w) \psi_{\alpha_0} (\varepsilon (w) a) \rho(w)^{-1} = \rho (w) \psi_{\alpha_0} (a)^{\varepsilon (w)} \rho (w)^{-1}.
$$
This is clearly a regular map, and then in view of (\ref{E:ARR102}), the fact that (\ref{E:ARR101}) holds for $\alpha_0$ implies that it also holds for $\alpha.$ Thus, it is enough to construct an algebraic ring $A$ with a ring homomorphism $f \colon R \to A$ having Zariski-dense image such that a regular map $\psi_{\alpha} \colon A \to H$ satisfying (\ref{E:ARR101}) exists for a single root of each length in $\Phi$. Furthermore, if all roots in $\Phi$ have the same length, then $\Phi$ contains a subsystem $\Phi'$ of type $A_2.$ Otherwise, $\Phi$ either contains a subsystem $\Phi'$ of type $B_2$, or $\Phi$ itself is of type $G_2,$ in which case we set $\Phi' = \Phi.$ In all three cases, $\Phi'$ contains a root of each length occuring in $\Phi$, so it follows from our previous remark that it is enough to construct an algebraic ring $A$ and a ring homomorphism $f \colon R \to A$ with $\overline{f(R)} = A$ such that $\psi_{\alpha} \colon A \to H$ satisfying (\ref{E:ARR101}) exists for a single root in $\Phi'$ of each length. Thus, we can assume without any loss of generality that $\Phi$ is of one of the types $A_2, B_2,$ or $G_2.$

Second, in each of these three cases, $A$ will be constructed as $A_{\alpha} \colon = \overline{\rho (e_{\alpha} (R))}$ for some root $\alpha \in \Phi.$ Letting $\ba_{\alpha} \colon A_{\alpha} \times A_{\alpha} \to A_{\alpha}$ denote the restriction of the product $H \times H \to H$ to $A_{\alpha}$, we observe that
$(A_{\alpha}, \ba_{\alpha})$ is a commutative algebraic subgroup of $H$ (in particular, it is an affine algebraic variety), and
%and if we let $\ba_{\alpha} \colon A_{\alpha} \times A_{\alpha} \to A_{\alpha}$ denote the restriction of the product $H \times H \to H$ to $A_{\alpha}$, then
$$
f_{\alpha} \colon (R, +) \to (A_{\alpha}, \ba_{\alpha}), \ \ \ t \mapsto \rho(e_{\alpha} (t))
$$
is a group homomorphism with Zariski-dense image. Multiplication $\bm \colon A \times A \to A$ will be defined in each case by {\it ad hoc} equations depending on the explicit form of the Steinberg relations. The verification of the fact that $(A, \ba, \bm)$ is an algebraic ring in all three cases relies on the following simple observation.

\begin{lemma}\label{L:ARR-1}
Let $A$ be an affine  variety equipped with two regular maps $\ba \colon A \times A \to A$ and $\bm \colon A \times A \to A$. Assume that $(A, \ba)$ is a commutative algebraic group and that there exists a homomorphism $f \colon R \to A$ of an abstract commutative unital ring $R$ into $A$ such that $\overline{f(R)} = A$ and
\begin{equation}\label{E:ARR201}
f(t_1 + t_2) = \ba (f(t_1), f(t_2)) \ \ \ {\rm and} \ \ \ f(t_1 t_2) = \bm (f(t_1), f(t_2))
\end{equation}
for all $t_1, t_2 \in R.$ Then $(A, \ba, \bm)$ is a commutative algebraic ring with identity.

\end{lemma}
\begin{proof}
Condition (I) of Definition 2.1 holds by our assumption. To verify (II), we observe that (\ref{E:ARR201}), in conjunction with the fact that multiplication in $R$ is associative, implies that the regular maps
$$
\beta_1 \colon A \times A \times A \to A, \ \ \ \beta_1 (x,y,z) = \bm (\bm (x,y), z),
$$
and
$$
\beta_2 \colon A \times A \times A \to A, \ \ \ \beta_2 (x,y,z) = \bm (x, \bm (y,z))
$$
coincide on the Zariski-dense subset $f(R) \times f(R) \times f(R) \subset A \times A \times A$. It follows that they coincide everywhere, yielding (II). All other conditions (including the fact that $1_A := f(1_R)$ is the identity element in $A$) are verified similarly.
\end{proof}

To complete the proof of the theorem, we will now consider separately the cases where $\Phi$ is of type $A_2, B_2,$ and $G_2.$

\vskip1mm

\noindent {\sc Case I:} $\Phi$ {\sc of Type} $A_2$ (cf. \cite{Kas}, Theorem 3)

\noindent We will use the standard realization of $\Phi$, described in \cite{Bour}, where the roots are of the form $\varepsilon_i - \varepsilon_j$,
with $i,j \in \{1, 2, 3 \}, i \neq j.$ We will write $e_{ij} (t)$ to denote $e_{\alpha} (t)$ for $\alpha = \varepsilon_i - \varepsilon_j.$ Set $\alpha = \varepsilon_1 - \varepsilon_3,$ and define $A$ to be $A_{\alpha} = \overline{\rho (e_{\alpha} (R))}.$ As we observed earlier, $(A, \ba)$ is a commutative algebraic subgroup of $H$, where $\ba \colon A \times A \to A$ is the restriction of the product in $H$ to $A$.
Furthermore, we let $f = f_{\alpha} \colon R \to A$ be the map defined by
$t \mapsto \rho (e_{\alpha} (t)).$ Clearly
\begin{equation}\label{E:ARR202}
\ba (f(t_1), f(t_2)) = f(t_1 + t_2)
\end{equation}
for all $t_1, t_2 \in R.$ To define $\bm$, we need the following elements
$$
w_{12} = e_{12} (1) e_{21} (-1) e_{12} (1) \ \ \ {\rm and} \ \ \ w_{23} = e_{23} (1) e_{32} (-1) e_{23} (1).
$$
It is easily checked that
\begin{equation}\label{E:ARR203}
w_{12}^{-1} e_{13}(r) w_{12} = e_{23} (r), \ \ \ w_{23} e_{13} (r) w_{23}^{-1} = e_{12} (r)
\end{equation}
and
\begin{equation}\label{E:ARR204}
[e_{12} (r), e_{23} (s)] = e_{13} (rs)
\end{equation}
for all $r, s \in R$, where $[g, h] = g h g^{-1} h^{-1}.$ Define the regular map $\bm \colon A \times A \to H$ by
$$
\bm (a_1, a_2) = [\rho(w_{23}) a_1 \rho(w_{23})^{-1}, \rho(w_{12})^{-1} a_2 \rho(w_{12})].
$$
It follows from relations (\ref{E:ARR203}) and (\ref{E:ARR204}) that
\begin{equation}\label{E:ARR205}
\bm (f(t_1), f(t_2)) = f(t_1 t_2).
\end{equation}
Then, in particular, $\bm (f(R) \times f(R)) \subset f(R),$ implying that $\bm (A \times A) \subset A$, and allowing us to view $\bm$ as a regular map $\bm \colon A \times A \to A.$ Using (\ref{E:ARR202}) and (\ref{E:ARR205}) and applying Lemma \ref{L:ARR-1}, we conclude that $(A, \ba, \bm)$ is a commutative algebraic ring with identity. Finally, by our construction, (\ref{E:ARR101}) obviously holds for the inclusion map $\psi_{\alpha} \colon A \to H$, as required.

\vskip2mm

\noindent {\sc Case II:} $\Phi$ {\sc of Type} $B_2$

\noindent We will use the realization of $\Phi$ described in \cite{Bour} as the set of vectors $\pm \varepsilon_1, \pm \varepsilon_2, \pm \varepsilon_1 \pm \varepsilon_2$ in $\R^2$, where $\varepsilon_1, \varepsilon_2$ is the standard basis of $\R^2.$ Set $\alpha = \varepsilon_1$ and $\beta = \varepsilon_1 + \varepsilon_2.$ As we remarked earlier, it is enough to construct an algebraic ring $A$ with a ring homomorphism $f \colon R \to A$ having Zariski-dense image, and regular maps $\psi_{\alpha} \colon A \to H$ and $\psi_{\beta} \colon A \to H$ satisfying (\ref{E:ARR101}). In fact, we will show that one can take $A = A_{\alpha}$ and $f = f_{\alpha}.$ As in the previous case, for addition $\ba$, we simply take the restriction to $A \times A$ of the product $H \times H \to H$. To define the multiplication map $\bm$, we need to work simultaneously with $A_{\alpha}$ and $A_{\beta}.$

%We use the realization of $\Phi$ described in \cite{Bour}, and denote
%the short roots of $\Phi$ by $\pm \varepsilon_i$ ($i = 1,2$) and the long roots by $\pm \varepsilon_1 \pm \varepsilon_2.$ Assume that $2 \in R^{\times}$ and let $S= A_{\varepsilon_1} = \overline{\rho (e_{\varepsilon_1} (R))}$ and $L= A_{\varepsilon_1+\varepsilon_2} = \overline{\rho(e_{\varepsilon_1 + \varepsilon_2} (R))}.$ As noted earlier, $S$ and $L$ are abelian subgroups of $H$, and we let $\ba_S \colon S \times S \to S$ and $\ba_L \colon L \times L \to L$ denote the restrictions of the product in $H$ to $S$ and $L$, respectively. Furthermore, consider the maps $\sigma \colon R \to S$, $t \mapsto  \rho(e_{\varepsilon_1} (t))$, and $\lambda \colon R \to L$,  $t \mapsto \rho (e_{\varepsilon_1 + \varepsilon_2} (t))$. It is easily seen that
%\begin{equation}\label{E:ARR206}
%\ba_S (\sigma (t_1), \sigma (t_2)) = \sigma (t_1 + t_2) \ \ \ {\rm and} \ \ \ \ba_L %(\lambda (t_1), \lambda (t_2)) = \lambda (t_1 + t_2)
%\end{equation}
%for all $t_1, t_2 \in R.$ Before defining the multiplication operations, we prove the following
\begin{lemma}\label{L:ARR-2}
There exists an isomorphism $\pi \colon A_{\alpha} \to A_{\beta}$ of algebraic varieties such that $\pi \circ f_{\alpha} = f_{\beta}.$
\end{lemma}
\begin{proof}
Define $\pi \colon A_{\alpha} \to H$ by
$$
\pi(x) = [x, f_{\varepsilon_2}(1/2)]
$$
(recall that by our assumption, $1/2 \in R$).
The commutator relation
\begin{equation}\label{E:ARR108}
[e_{\varepsilon_1} (s), e_{\varepsilon_2}(t)] = e_{\varepsilon_1 + \varepsilon_2} (2st)
\end{equation}
shows that $\pi(\rho (e_{\varepsilon_1}(r))) = \rho (e_{\varepsilon_1 + \varepsilon_2} (r))$, i.e.
\begin{equation}\label{E:ARR220}
\pi (f_{\alpha} (r)) =f_{\beta} (r)
\end{equation}
for any $r \in R.$ In particular, $\pi (f_{\alpha} (R)) \subset f_{\beta} (R),$ and since $\pi$ is obviously regular, we obtain that $\pi (A_{\alpha}) \subset A_{\beta}.$
The inverse map $\nu \colon A_{\beta} \to A_{\alpha}$ is constructed as follows. Recall the following standard notations from (\cite{Stb1}, Chapter 3): for $t \in R^{\times}$ and any $\gamma \in \Phi,$
$$
w_{\gamma} (t) = e_{\gamma} (t) e_{-\gamma} (-t^{-1}) e_{\gamma} (t) \ \ \ {\rm and} \ \ \ h_{\gamma} (t) = w_{\gamma} (t) w_{\gamma} (-1).
$$
Now for $y \in A_{\beta}$, set
$$
\nu (y) = \rho(h_{\varepsilon_1 + \varepsilon_2}(1/2))[y, f_{-\varepsilon_2} (1)][y, f_{-\varepsilon_2} (-1) ]^{-1} \rho(h_{\varepsilon_1 + \varepsilon_2}(1/2))^{-1}.
$$
This clearly defines a regular map $\nu \colon A_{\beta} \to H.$
By direct computation, we have
$$
[e_{\varepsilon_1 + \varepsilon_2} (t), e_{- \varepsilon_2} (s)] = e_{\varepsilon_1}(ts) e_{\varepsilon_1 - \varepsilon_2} (-ts^2)
$$
for all $s, t \in R.$ Letting $s = \pm 1$ and using the fact that $e_{\varepsilon_1} (s)$ and $e_{\varepsilon_1 - \varepsilon_2} (t)$ commute, we obtain
$$
[e_{\varepsilon_1 + \varepsilon_2} (t), e_{- \varepsilon_2} (1)] [e_{\varepsilon_1 + \varepsilon_2} (t), e_{-\varepsilon_2} (-1)]^{-1} = e_{\varepsilon_1} (2t).
$$
Combining this with the relation
\begin{equation}\label{E:ARR109}
h_{\varepsilon_1 + \varepsilon_2} (1/2) e_{\varepsilon_1} (t) h_{\varepsilon_1 + \varepsilon_2} (1/2)^{-1} = e_{\varepsilon_1} (t/2).
\end{equation}
shows that
\begin{equation}\label{E:ARR221}
\nu (f_{\beta}(r)) = f_{\alpha}(r),
\end{equation}
hence $\nu (f_{\beta} (R)) \subset f_{\alpha} (R)$ and $\nu (A_{\beta}) \subset A_{\alpha}.$
It follows from (\ref{E:ARR220}) and (\ref{E:ARR221}) that $\nu \circ \pi$ and $\pi \circ \nu$ are the identity maps of $f_{\alpha} (R)$ and $f_{\beta} (R)$, respectively, so $\nu = \pi^{-1}.$ Also, by (\ref{E:ARR220}), $\pi$ is as required.
\end{proof}

We now define $\bm \colon A \times A \to H$, where $A = A_{\alpha},$ by
$$
\bm (u, v) = \nu ([u, v'']),
$$
where for $v \in A$ we set
$$
v'= \rho(h_{\beta}(1/2)) v \rho(h_{\beta}(1/2))^{-1}
$$
and
$$
v''= \rho(w_{-\varepsilon_1 + \varepsilon_2}(1)) v' \rho(w_{-\varepsilon_1 + \varepsilon_2} (1))^{-1},
$$
and $\nu$ is the inverse of the map $\pi$ constructed in Lemma \ref{L:ARR-2}.
The relations (\ref{E:ARR108}) and (\ref{E:ARR109}), combined with
$$
w_{-\varepsilon_1 + \varepsilon_2}(1) e_{\varepsilon_1} (t) w_{-\varepsilon_1 + \varepsilon_2} (1)^{-1} = e_{\varepsilon_2} (t)
$$
show that
\begin{equation}\label{E:ARR110}
\bm (f_{\alpha}(r), f_{\alpha}(s)) = f_{\alpha} (rs).
\end{equation}
In particular, $\bm_S (f_{\alpha} (R) \times f_{\alpha} (R)) \subset f_{\alpha} (R),$ implying that $\bm (A \times A) \subset A$, and allowing us to  view $\bm$ as a regular map $\bm \colon A \times A \to A.$ The fact that $f_{\alpha}$ is additive and satisfies (\ref{E:ARR110}) enables us to apply Lemma \ref{L:ARR-1} to conclude that $(A, \ba, \bm)$ is an algebraic ring. Furthermore, the inclusion map $\psi_{\alpha} \colon A \to H$ is as required. Finally, we define $\psi_{\beta} \colon A \to H$ to be $\psi_{\beta} = \pi \circ \psi_{\alpha}.$ Then it follows from Lemma \ref{L:ARR-2} that $\psi_{\beta}$ satisfies (\ref{E:ARR101}).

We note that using the maps $\pi$ and $\nu$ introduced in the proof of Lemma \ref{L:ARR-2}, one can directly define a multiplication $\bm_{\beta}$ on $A_{\beta}$ by setting
$$
\bm_{\beta} (a,b) = \pi (\bm (\nu (a), \nu(b))),
$$
so that $A_{\beta}$ also becomes an algebraic ring.

\vskip2mm

\noindent {\sc Case III:} $\Phi$ of type $G_2$

\noindent We will use the realization of $\Phi$ described in \cite{CK}: one can pick a system of simple roots $\{k, c \}$ in $\Phi$, where $k$ is long and $c$ is short, and then the long roots of $\Phi$ are
$\pm k, \pm (3c + k), \pm (3c + 2k),$ and the short roots are $\pm c, \pm (c+k), \pm (2c + k).$ Set $\alpha = k$ and $\beta = 2c + k.$ Since the long roots of $\Phi$ form a closed subsystem of type $A_2$
(cf. \cite{H1},  19.4), it follows that $A = A_{\alpha}$ is an algebraic ring, $f =f_{\alpha}$ is a ring homomorphism $R \to A$ with Zariski-dense image, and (\ref{E:ARR101}) holds if  $\psi_{\alpha} \colon A \to H$ is the inclusion map.
To construct $\psi_{\beta} \colon A \to H$ that also satisfies (\ref{E:ARR101}), we need the following analogue of Lemma \ref{L:ARR-2}.
\begin{lemma}\label{L:ARR-3}
There exists an isomorphism of algebraic varieties $\varkappa \colon A_{\alpha} \to A_{\beta}$ such that
$\varkappa \circ f_{\alpha} = f_{\beta}.$
\end{lemma}
\begin{proof}
The following explicit forms of the Steinberg commutator relations were established in
(\cite{CK}, Theorem 1.1):
\begin{equation}\label{E:ARR160}
[e_{k}(s), e_{c} (t)] = e_{c+k} (\varepsilon_1 st) e_{2c + k} (\varepsilon_2 st^2) e_{3c + k} (\varepsilon_3 st^3) e_{3c + 2k} (\varepsilon_4 s^2 t^3),
\end{equation}
\begin{equation}\label{E:ARR161}
[e_{c+k} (s), e_{2c+k}(t)] = e_{3c+2k} (3 \varepsilon_5 st),
\end{equation}
\begin{equation}\label{E:ARR162}
[e_{3c+k}(s), e_{k} (t)] = e_{3c + 2k} (\varepsilon_6 st),
\end{equation}
where $\varepsilon_i = \pm 1.$\footnote{In the proof of Theorem 28 in \cite{Stb1}, the precise signs are computed for certain pairs of roots, but these may change under the action of the Weyl group.}
Using (\ref{E:ARR160}), we obtain
$$
[e_k(s), e_c(1)] [e_k (s), e_c(-1)] =
$$
$$
=e_{c+k} (\varepsilon_1 s) e_{2c+k} (\varepsilon_2 s) e_{3c + k} (\varepsilon_3 s) e_{3c+2k} (\varepsilon_4 s^2) e_{c+k} (-\varepsilon_1 s) e_{2c+k} (\varepsilon_2 s) e_{3c + k} (-\varepsilon_3 s) e_{3c+2k} (-\varepsilon_4 s^2).
$$
Since the terms $e_{3c+k} (-\varepsilon_3 s)$ and $e_{3c + 2k} (-\varepsilon_4 s^2)$ commute with all other terms, the last expression reduces to
$$
e_{c+k} (\varepsilon_1 s) e_{2c +k} (\varepsilon_2 s) e_{c+k}(-\varepsilon_1 s) e_{2c + k} (\varepsilon_2 s),
$$
which, using (\ref{E:ARR161}), can be written in the form
$$
e_{3c + 2k} (3 \varepsilon_5 \varepsilon_1 \varepsilon_2 s^2) e_{2c + k} (2 \varepsilon_2 s).
$$
Let $w_1 = w_c (1) \in G(R)^+$. Then it follows from (\cite{St1}, 3.8 relation (R4)), that $w_1 e_k ( s) w_1^{-1} = e_{ 3c+k} (\varepsilon s),$ where $\varepsilon = \pm 1.$ Set $\gamma = 3c+2k$ and $\delta = c.$ Then
$$
< k, \gamma > = 1 = <2c+k, \delta>,
$$
so for any $r \in R^{\times}$
$$
h_{\gamma} (r) e_k (s) h_{\gamma} (r)^{-1} = e_k (r s)
$$
and
$$
h_{\delta} (r) e_{2c+k} (s) h_{\delta} (r)^{-1} = e_{2c +k} (rs).
$$
by (\cite{St1}, 3.8 relation (R6)). Since by our assumption $2 \in R^{\times},$ we can define $\varkappa' \colon A_{\alpha} \to H$ by
$$
\varkappa' (u) = [\rho(w_1) u \rho(w_1)^{-1}, \rho(h_{\gamma} (-3 \varepsilon \varepsilon_1 \varepsilon_2 \varepsilon_5 \varepsilon_6)) u \rho (h_{\gamma} (-3 \varepsilon \varepsilon_1 \varepsilon_2 \varepsilon_5 \varepsilon_6))^{-1}] [u, f_c (1)][u, f_c (-1)]
$$
and then consider $\varkappa \colon A_{\alpha} \to H$ given by
$$
\varkappa (u) = \rho (h_{\delta} (\varepsilon_2/2)) \varkappa' (u) \rho (h_{\delta} (\varepsilon_2/ 2))^{-1}.
$$
Then $\varkappa$ is obviously a regular map, and the above computations imply that $\varkappa \circ f_{\alpha} = f_{\beta},$ hence $\varkappa (f_{\alpha} (R)) \subset f_{\beta} (R)$ and $\varkappa (A_{\alpha}) \subset A_{\beta}.$

The inverse map $\theta \colon A_{\beta} \to A_{\alpha}$ is constructed as follows. Applying an appropriate element of the Weyl group to (\ref{E:ARR161}), we obtain
$$
[e_{-c} (s), e_{c+k} (t)] = e_k (3 \varepsilon_7 st),
$$
with $\varepsilon_7 = \pm 1.$ Set $w_2 = w_{3c+k}(1).$ Then $w_2 e_{2c+k} (s) w_2^{-1} = e_{-c} (\varepsilon ' s),$ with $\varepsilon' = \pm 1.$ Since $3 \in R^{\times},$ we can define a regular map $\theta \colon A_{\beta} \to H$ by
$$
\theta (u) = \rho(h_{\alpha}(\varepsilon' \varepsilon_7/3)) [\rho(w_2) u \rho (w_2)^{-1}, f_{c+k} (1)] \rho (h_{\alpha}(\varepsilon' \varepsilon_7/3))^{-1}.
$$
Then by our construction $\theta \circ f_{\beta} = f_{\alpha},$ implying that $\theta (A_{\beta}) \subset A_{\alpha}.$ The compositions $\theta \circ \varkappa$ and $\varkappa \circ \theta$ are clearly the identity maps of $f_{\alpha} (R)$ and $f_{\beta} (R)$, respectively, hence $\theta = \varkappa^{-1}.$
\end{proof}

It follows from the lemma that $\psi_{\beta} := \varkappa \circ \psi_{\alpha}$ satisfies (\ref{E:ARR101}), as required.

\vskip2mm

As in Case II, we note that using the maps $\varkappa$ and $\theta$ introduced in the proof of Lemma \ref{L:ARR-3}, we can define a multiplication map $\bm_{\beta}$ on $A_{\beta}$ by setting
$$
\bm_{\beta} (u, v) = \varkappa (\bm (\theta(u), \theta(v))),
$$
and then $A_{\beta}$ becomes an algebraic ring.

\vskip2mm

\noindent {\bf Remark 3.5.} The referee has suggested the following uniform interpretation of the construction of the isomorphism between $A_{\alpha}$ and $A_{\beta}$ in Lemmas \ref{L:ARR-2} and \ref{L:ARR-3}: first one finds a word $S_{\alpha, \beta}$ in the free group such that $S_{\alpha, \beta} (e_{\alpha} (t), g_1, \dots, g_r) = e_{\beta} (t)$ for all $t \in R$, where $g_1, \dots, g_r$ are some fixed elements of $G(R)$; then, the isomorphism of algebraic varieties $A_{\alpha} \to A_{\beta}$ is given by $x \mapsto S_{\alpha, \beta} (x, \rho(g_1), \dots, \rho(g_r))$. Using these isomorphisms, one can then give a similar interpretation of the definition of the product operation.

%By our construction, $\bm_{S'} (\sigma' (r), \sigma' (s)) = \sigma' (rs)$, hence $\bm_{S'} (S' \times S') \subset S'.$ Applying Lemma \ref{L:ARR-1} as before, we conclude that $(S', \ba_{S'}, \bm_{S'})$ is a commutative algebraic ring with identity.
%The same argument as in Case II shows that $\kappa \colon L' \to S'$ is an isomorphism of algebraic rings.

%Thus, as in Case II, we may consider $S'$ and $L'$ as a single algebraic ring $A$ and the maps $\sigma'$ and $\lambda'$ as a single ring homomorphism $f \colon R \to A.$ Finally, we define $\psi_{k} \colon A \to H$ and $\psi_{2c+k} \colon A \to H$ to be the inclusion maps.

\vskip5mm

\section{Steinberg groups}\label{S:StG}

The next step in the proof of the Main Theorem involves Steinberg groups, so we begin by recalling their definition. As in $\S$ \ref{S:ARR}, let $\Phi$ be a reduced irreducible root system of rank $\geq 2$, let $G$ be the corresponding universal Chevalley group scheme, and $e_{\alpha} \colon \mathbb{G}_a \to G$ be the one-parameter subgroup associated with $\alpha \in \Phi.$ Suppose $S$ is a commutative ring. It is well-known (cf. \cite{Stb1}, Chapter 3) that the elements $e_{\alpha} (s)$, for $\alpha \in \Phi$ and $s \in S$, satisfy the following relations:
\begin{equation}\label{E:StG103}
e_{\alpha} (s) e_{\alpha} (t) = e_{\alpha} (s+t)
\end{equation}
for all $s,t \in S$ and all $\alpha \in \Phi,$ and
\begin{equation}\label{E:StG104}
[e_{\alpha} (s), e_{\beta} (t)] = \prod e_{i \alpha + j \beta} (N^{i,j}_{\alpha, \beta} s^i t^j),
\end{equation}
for all $s,t \in S$ and all $\alpha, \beta \in \Phi,$ $\beta \neq - \alpha,$
where the product is taken over all roots of the form $i \alpha + j \beta,$ $i, j \in \Z^+$, listed in an arbitrary (but {\it fixed}) order, and the $N^{i,j}_{\alpha, \beta}$ are integers depending only on  $\Phi$ and the order of the factors in (\ref{E:StG104}), but not on the ring $S$.

\vskip2mm

\noindent {\bf Definition 4.1.} The {\it Steinberg group} $\tilde{G}(S)$ is the group with generators $\tilde{x}_{\alpha} (t)$, for all $t \in S$ and $\alpha \in \Phi$, subject to the relations
\vskip1mm

(R1) $\tilde{x}_{\alpha}(s) \tilde{x}_{\alpha}(t) = \tilde{x}_{\alpha} (s+t)$

\vskip1mm

(R2) $[\tilde{x}_{\alpha} (s), \tilde{x}_{\beta} (t)] = \prod \tilde{x}_{i \alpha + j \beta} (N^{i,j}_{\alpha, \beta} s^i t^j)$,

\vskip1mm

\noindent where $N^{i,j}_{\alpha, \beta}$ are the same integers as in (\ref{E:StG104}).

\vskip2mm

It follows from the definition and the relations (\ref{E:StG103}) and (\ref{E:StG104}) that there exists a surjective group homomorphism
$$
\pi_S \colon \tilde{G} (S) \to G(S)^+, \ \ \ \tilde{x}_{\alpha} (t) \mapsto e_{\alpha} (t).
$$
Let
$$
K_2 (\Phi, S) = \ker \pi_S.
$$
We note that the pair $(\tilde{G}(S), \pi_S)$ is functorial in the following sense: given a homomorphism of commutative rings
$f \colon S \to T$, there is a commutative diagram of group homomorphisms
$$
\xymatrix{ \tilde{G}(S) \ar[d]_{\pi_S} \ar[r]^{\tilde{F}} & \tilde{G}(T) \ar[d]^{\pi_T} \\ G(S)^+ \ar[r]^{F} & G(T)^+}
$$
%$$
%\begin{array}{ccc} \tilde{G}(S) & \stackrel{\tilde{F}}{\longrightarrow} & \tilde{G}(T) \\ \pi_S \downarrow & & \downarrow \pi_{T} \\ G(S)^+ & \stackrel{F}{\longrightarrow} & G(T)^+ \end{array}
%$$
where $F$ and $\tilde{F}$ are the natural homomorphisms induced by $f$, i.e. that map the generators as follows:
$F \colon e_{\alpha} (t) \mapsto e_{\alpha} (f(t))$ and $\tilde{F} \colon \tilde{x}_{\alpha} (t) \mapsto \tilde{x}_{\alpha} (f(t)).$

The goal of this section is twofold: first, for a given representation $\rho \colon G(R)^+ \to GL_n (K)$ and the associated algebraic ring $A$ (cf. $\S$ \ref{S:ARR}), we want to construct a representation $\tilde{\tau} \colon \tilde{G} (A) \to GL_n (K)$ of the corresponding Steinberg group such that $\tilde{x}_{\alpha} (a) \mapsto \psi_{\alpha} (a)$ for all $a \in A$, where $\psi_{\alpha}$ is the regular map from Theorem \ref{T:ARR-1}; second, we would like to investigate the structure of $\tilde{G}(A).$

\addtocounter{thm}{1}

\begin{prop}\label{P:StG-1}
Let $(\Phi, R)$ be a nice pair, $K$ an algebraically closed field, and $\rho \colon G(R)^+ \to GL_n (K)$ a representation. Furthermore, let $A$ and $f \colon R \to A$ be the algebraic ring and ring homomorphism associated to $\rho$ that were constructed in Theorem \ref{T:ARR-1}. Then there exists a group homomorphism $\tilde{\tau} \colon \tilde{G}(A) \to H \subset GL_n (K)$ such that $\tilde{\tau} \colon \tilde{x}_{\alpha} (a) \mapsto \psi_{\alpha} (a)$ for all $a \in A$ and all $\alpha \in \Phi$. Consequently, $\tilde{\tau} \circ \tilde{F} = \rho \circ \pi_R,$ where $\tilde{F} \colon \tilde{G}(R) \to \tilde{G}(A)$ is the homomorphism induced by $f.$
\end{prop}
\begin{proof}
To establish the existence of $\tilde{\tau},$ we need to show that relations (R1) and (R2) are satisfied if the $\tilde{x}_{\alpha} (a)$'s are replaced by $\psi_{\alpha} (a)$'s. First, let $a = f(s), b= f(t)$, with $s, t \in R.$ Since $\rho$ is a homomorphism and by Theorem \ref{T:ARR-1}, $\rho \circ e_{\alpha} = \psi_{\alpha} \circ f$, we have, in view of (\ref{E:StG103}), that
$$
\psi_{\alpha} (a) \psi_{\alpha} (b) = \rho (e_{\alpha} (t) e_{\alpha} (s)) = \rho (e_{\alpha} (t+s)) = \psi_{\alpha} (a + b).
$$
Thus, the two regular maps $A \times A \to H$ given by
$$
(a, b) \mapsto \psi_{\alpha} (a) \psi_{\alpha} (b) \ \ \ {\rm and} \ \ \ (a,b) \mapsto \psi_{\alpha} (a + b)
$$
coincide on $f(R) \times f(R).$ The Zariski density of the latter in $A \times A$ implies that they coincide everywhere, yielding (R1).

Similarly, using (\ref{E:StG104}), we see that the two regular maps $A \times A \to H$ defined by
$$
(a, b) \mapsto [ \psi_{\alpha} (a), \psi_{\beta} (b)] \ \ \ {\rm and} \ \ \ (a,b) \mapsto \prod \psi_{i \alpha + j \beta} (N^{i,j}_{\alpha, \beta} a^i b^j)
$$
coincide on $f(R) \times f(R),$ hence on $A \times A$, and (R2) follows. Finally, the maps $\tilde{\tau} \circ \tilde{F}$ and $\rho \circ \pi_R$ both send $\tilde{x}_{\alpha} (s),$ $s \in R$, to $\psi_{\alpha} (f(s)) = \rho (e_{\alpha} (s)) = (\rho \circ \pi_R) (\tilde{x}_{\alpha} (s)),$ so they coincide on $\tilde{G} (R).$

\end{proof}

Thus, we obtain that the diagram formed by the solid arrows in
\begin{equation}\label{D:StG-1}
\xymatrix{ \tilde{G}(R) \ar[d]_{\pi_R} \ar[r]^{\tilde{F}} & \tilde{G}(A) \ar[rrdd]^{\tilde{\tau}} \ar[d]^{\pi_A} \\ G(R) \ar[rrrd]_{\rho} \ar[r]^{F} & G(A) \ar@{.>}[rrd]^{\tau} \\ & & & H \\}
\end{equation}
commutes. The crucial part in the proof of the Main Theorem is basically to show that
$\tilde{\tau}$ descends to a homomorphism $\tau \colon G(A) \to H$ which makes the whole diagram commutative --- the precise statement is somewhat more technical and will be given in \S \ref{S:FI}.\footnote{Note that until \S 6, we treat $G(A)$ as an {\it abstract} group; the structure of an algebraic group on $G(A)$ will be discussed in \S 6, where it will also be shown that a map related to $\tau$ is actually a morphism of algebraic groups.}
Determining when such a $\tau$ exists obviously requires information about
$K_2 (\Phi, A) = \ker \pi_A$, which we derive from results of M.~Stein \cite{St2}, describing $K_2 (\Phi, S)$ for a commutative semilocal ring $S$. To give precise formulations, we first need to recall some standard notations.
For $\alpha \in \Phi$ and $u \in S^{\times}$, we define the following elements of  $\tilde{G}(S)$:
\begin{equation}\label{E:StG-20}
\tilde{w}_{\alpha} (u) = \tilde{x}_{\alpha} (u) \tilde{x}_{-\alpha} (-u^{-1}) \tilde{x}_{\alpha} (u) \ \ \ {\rm and} \ \ \  \tilde{h}_{\alpha}(u) = \tilde{w}_{\alpha} (u) \tilde{w}_{\alpha} (-1).
\end{equation}
Notice that the elements $w_{\alpha} (u) = \pi_S (\tilde{w}_{\alpha} (u))$ and $h_{\alpha} (u) = \pi_S ( \tilde{h}_{\alpha} (u))$ coincide with the ones introduced in the proof of Lemma \ref{L:ARR-2}. Now,
for $u, v \in S^{\times}$, the {\it Steinberg symbol} $\{ u, v \}_{\alpha}$ is defined as
$$
\{ u, v \}_{\alpha} = \tilde{h}_{\alpha} (uv) \tilde{h}_{\alpha} (u)^{-1} \tilde{h}_{\alpha} (v)^{-1}.
$$
We note that since the elements $h_{\alpha} (t)$ are multiplicative in $t$ (\cite{Stb1}, Lemma 28), all Steinberg symbols are contained in $K_2 (\Phi, S).$ Moreover, by
(\cite{St2}, Proposition 1.3(a)), the Steinberg symbols lie in the center of $\tilde{G}(S).$ Following Stein \cite{St2}, given a nonempty subset $P \subset S$, we let $\Z [P]$ denote the subring of $S$ generated by $P.$
With these notations, we have
\begin{thm}\label{T:StG-1}{\rm (cf. \cite{St2}, Theorem 2.13)}
Let $\Phi$ be a reduced irreducible root system of rank $\geq 2.$ If $S$ is a commutative semilocal ring with $S = \Z[S^{\times}]$, then
$K_2 (\Phi, S)$ is the central subgroup of $\tilde{G}(S)$ generated by the Steinberg symbols $\{ u, v \}_{\alpha}$, with $u,v \in S^{\times}$, for {\rm any} fixed long root $\alpha.$
\end{thm}
Now our results in $\S$ \ref{S:AR} on the structure of algebraic rings immediately yield
\begin{cor}\label{C:StG-1}
Let $A$ be a connected commutative algebraic ring. Then $K_2 (\Phi, A)$ is a central subgroup of $\tilde{G}(A)$ generated by the Steinberg symbols $\{ u, v \}_{\alpha}$ for $u, v \in A^{\times}$ and any fixed long root $\alpha.$
\end{cor}
\begin{proof}
Since $A$ is connected, we have $A = A^{\times} - A^{\times}$ (Corollary \ref{C:AR-1}), hence $A = \Z [A^{\times}]$. Furthermore, by Lemma \ref{L:AR-1}, any commutative algebraic ring is semilocal as an abstract ring. Thus, the statement follows directly from Theorem \ref{T:StG-1}.
\end{proof}

The centrality of $K_2 (\Phi, A)$ is critical for the proof of our main results. We first note the following finiteness statement.
\begin{prop}\label{P:StG-2}
Let $S$ be a finite commutative ring and $\Phi$ a reduced irreducible root system of rank $\geq 2.$ Then $\tilde{G}(S)$ and $K_2 (\Phi, S)$ are finite.
\end{prop}
\begin{proof}
Even though this result is probably known, we failed to find a direct reference, so we give a complete proof.
First observe that since $G$ is a group scheme of finite type over $\Z$ and $S$ is finite, the group $G(S)$ is finite, so the finiteness of $\tilde{G}(S)$ is equivalent to that of $K_2 (\Phi, S).$
Now, being finite, $S$ is artinian, and consequently is a finite direct product of local artinian rings
$$
S = \prod_{i=1}^r S_i
$$
(cf. \cite{At}, Theorem 8.7). Then by (\cite{St2}, Lemma 2.12), we have
$$
K_2 (\Phi, S) = \prod_{i=1}^r K(\Phi, S_i),
$$
which reduces the argument to the case of $S$ a local ring. Since the condition $S = \Z[S^{\times}]$ holds automatically for a local ring, we obtain from Theorem \ref{T:StG-1} that $K_2 (\Phi, S)$ is a central subgroup of $\tilde{G}(S).$ Now, in the case that
$G(S)$ and $\tilde{G}(S)$ are perfect groups (i.e. coincide with their commutator subgroups), our result follows from the fact that the universal central extension of a finite group is finite (in other words, the Schur multiplier of a finite group is finite). In the general case, we will imitate the proof of the finiteness of the Schur multiplier. More precisely, we consider the exact sequence
\begin{equation}\label{E:StG101}
1 \to K_2 (\Phi, S) \to \tilde{G}(S) \to G(S)^+ \to 1.
\end{equation}
and the corresponding initial segment of the Hochschild-Serre spectral sequence
\begin{equation}\label{E:StG102}
H^1 (\tilde{G}(S), \Q/ \Z) \to H^1 (K_2 (\Phi, S), \Q/ \Z)^{\tilde{G}(S)} \to H^2 (G(S)^+, \Q/ \Z),
\end{equation}
where all groups act trivially on $\Q/ \Z.$ Since $\tilde{G}(S)$ is finitely generated, with every generator having finite order, the group
$$
\tilde{G}(S)^{\rm ab} = \tilde{G}(S)/ [ \tilde{G}(S), \tilde{G}(S)]
$$
is finite. Since $H^1 (\tilde{G}(S), \Q/ \Z)$ is simply the dual group of $\tilde{G}(S)^{\rm ab},$ it is also finite.

Next, let us show that $H^2 (G(S)^+, \Q/ \Z)$ is finite. Let $n = \vert G(S)^+ \vert.$ The short exact sequence
$$
0 \to \mu_n \to \Q / \Z \stackrel{\times n}{\longrightarrow} \Q/ \Z \to 0,
$$
where $\mu_n = \frac{1}{n} \Z/ \Z$ and $\times n$ denotes multiplication by $n$, gives rise to the following exact sequence of cohomology groups
$$
H^2 (G(S)^+, \mu_n) \to H^2 (G(S)^+, \Q/ \Z) \stackrel{ \times n}{\longrightarrow} H^2(G(S)^+, \Q / \Z).
$$
It is well known that $H^2 (G(S)^+, \Q / \Z)$ is annihilated by multiplication by $n$ (see, for example, \cite{GS}, Corollary 3.3.8), so $H^2 (G(S)^+, \mu_n)$ surjects onto $H^2 (G(S)^+, \Q / \Z).$ On the other hand, $H^2 ( G(S)^+, \mu_n)$ is obviously finite, and the finiteness of $H^2 (G(S)^+, \Q/ \Z)$ follows.

Now we conclude that the middle term in (\ref{E:StG102}) is finite. Since, as we noted above, $K_2 (\Phi, S)$ is central in $\tilde{G}(S)$, this term coincides with the dual of $K_2 (\Phi, S).$ So, its finiteness implies that of $K_2 (\Phi, S)$, as required.
\end{proof}

Turning now to algebraic rings, we obtain

\begin{prop}\label{P:StG-3}
Suppose $A$ is a commutative algebraic ring over an algebraically closed field $K$ and $\Phi$ is a reduced irreducible root system of rank $\geq 2.$ Assume that $A$ satisfies {\rm (FG)} if $\text{char} \ K = p>0$.
%and that $A$ has at most one residue field isomorphic to $\mathbb{F}_2$ if $\Phi$ contains roots of different lengths.
Then $\tilde{G}(A) = \tilde{G} (A^{\circ}) \times P,$ where $P$ is a finite group.
\end{prop}
\begin{proof}
Applying Proposition \ref{P:AR-2} if char $K = 0$ and Proposition \ref{P:AR-4} if char $K > 0$ (observe that the latter proposition applies since $A$ is assumed to satisfy (FG)), we conclude that $A = A^{\circ} \times C$, where $C$ is a finite ring.
So, by (\cite{St2}, Lemma 2.12), we have
$$
\tilde{G}(A) = \tilde{G}(A^{\circ}) \times \tilde{G}(C),
$$
and by Proposition \ref{P:StG-2}, $P := \tilde{G} (C)$ is finite.
%For this, by Proposition \ref{P:StG-2},
%it is enough to show that $K_2 (\Phi, C)$ is central, which follows from Theorem \ref{T:StG-2} if all roots have the same length, and from Theorem \ref{T:StG-1} and Lemma \ref{L:StG-1} in all other cases.
\end{proof}
\begin{cor}\label{C:StG-2}
%Let $(\Phi, R)$ be a nice pair.
Suppose $R$ is an abstract commutative ring and $A$ is an algebraic ring over a field $K$ such that there exists an abstract ring homomorphism $f \colon R \to A$ with Zariski-dense image. Assume moreover that $R$ is noetherian if char $K > 0.$ Then $\tilde{G} (A) = \tilde{G}(A^{\circ}) \times P,$ where $P$ is a finite group.
\end{cor}
\begin{proof}
Note that if char $K > 0$, the fact that $R$ is noetherian implies that $A$ satisfies (FG) (see Lemma \ref{L:AR-3}).
%Next, if $\Phi$ contains root of different length, then $2 \in R^{\times}$ as $(\Phi, R)$ is nice.
So our assertion follows immediately from
%Corollary \ref{C:StG-3} and
Proposition \ref{P:StG-3}.
\end{proof}

\vskip1.5mm

\noindent {\bf Remark 4.8.} Using Lemma \ref{L:AR-5} in place of Propositions \ref{P:AR-4} and \ref{P:AR-2} in the above arguments, one can show that for {\it any} commutative algebraic ring $A$, one has the following direct product decomposition
\begin{equation}\label{E:411}
\tilde{G}(A) = \tilde{G}(A') \times P',
\end{equation}
where $A'$ is the subring (with identity) of $A$ consisting of all unipotent elements and $P'$ is a finite group. This can be viewed as a way to circumvent condition (FG) in the above statements. However, the structure of $A'$, and hence of $\tilde{G}(A')$, is in general difficult to describe, so the usefulness of (\ref{E:411}) in the investigation of (BT) is questionable.

\addtocounter{thm}{1}

\begin{cor}\label{C:StG-4}
Let $(\Phi, R)$ be a nice pair. Assume that there is an integer $m \geq 1$ such that $m R = \{ 0 \}.$ Let $K$ be an algebraically closed field such that char $K$ does not divide $m$, and suppose moreover that $R$ is noetherian if char $K > 0.$ If $\rho \colon G(R)^+ \to GL_n (K)$ is an abstract representation, then $\rho (G(R)^+)$ is finite.
\end{cor}
\begin{proof}
Let $A$ and $f \colon R \to A$ be the algebraic ring and ring homomorphism associated to $\rho$ constructed in Theorem \ref{T:ARR-1}. Then $mA = \{0 \}$ as $m f(R) = \{0 \}$ and $f(R)$ is Zariski-dense in $A$. In particular $m A^{\circ} = \{0 \}.$ Now recall that by Proposition \ref{P:AR-20}, $A^{\circ} / J$ is a $K$-algebra, where $J$ denotes the Jacobson radical of $A^{\circ}.$ Then our assumption that $m$ is not divisible by char $K$ forces $A^{\circ}/ J = \{ 0 \}$,  hence $A^{\circ} = \{ 0 \}.$ Therefore, by Corollary \ref{C:StG-2}, $\tilde{G}(A) = P,$ a finite group. On the other hand, by Proposition \ref{P:StG-1}, the representation $\rho$ factors through $\tilde{G}(A),$ which proves the finiteness of $\rho (G(R)^+).$
\end{proof}

\vskip5mm

\section{Passage to a subgroup of finite index}\label{S:FI}

As we remarked earlier, following Proposition \ref{P:StG-1}, our general strategy for
completing the proof of the Main Theorem is to construct a homomorphism $\tau \colon G(A) \to H$ which makes the diagram (\ref{D:StG-1}) commute.
We begin this section by making that idea more precise.
So, as in previous sections, suppose $\Phi$ is a reduced irreducible root system of rank $\geq 2$ and $R$ is a commutative ring such that $(\Phi, R)$ is a nice pair. Let $K$ be an algebraically closed field, and assume that $R$ is noetherian if char $K > 0.$
Furthermore, let $G$ be the universal Chevalley-Demazure group scheme of type $\Phi$, and let $\rho \colon G(R)^+ \to GL_n (K)$ be a representation. By Theorem 3.1, we can associate to $\rho$ an algebraic ring $A$ and a ring homomorphism $f \colon R \to A$ with Zariski-dense image. Moreover,
our assumptions imply that we can write $A = A^{\circ} \oplus C$, with $C$ a finite ring (in fact, a finite quotient of $R$) --- see Propositions \ref{P:AR-4} and \ref{P:AR-2}.
Now recall that by Proposition \ref{P:StG-1}, there exists a group homomorphism $\tilde{\tau} \colon \tilde{G}(A) \to GL_n (K)$ such that $\tilde{\tau} \circ \tilde{F} = \rho \circ \pi_R,$ where $\tilde{F} \colon \tilde{G}(R) \to \tilde{G}(A)$ is the homomorphism of Steinberg groups induced by $f$, and $\pi_R \colon \tilde{G}(R) \to G(R)$ is the canonical homomorphism. On the other hand, by Corollary \ref{C:StG-2}, $\tilde{G}(A) = \tilde{G}(A^{\circ}) \times P,$ where $P = \tilde{G}(C)$ is a finite group. So, $\tilde{\Gamma} := \tilde{F}^{-1} (\tilde{G} (A^{\circ}))$ and $\Gamma := \pi_R(\tilde{\Gamma})$ are subgroups of finite index in $\tilde{G}(R)$ and $G(R)^+$, respectively, and moreover $F(\Gamma) \subset G(A^{\circ}).$ Letting $\tilde{\sigma}$ denote the restriction of $\tilde{\tau}$ to $\tilde{G}(A^{\circ})$, we obtain that
the solid arrows in
\begin{equation}\label{D:FI-1}
\xymatrix{ \tilde{\Gamma} \ar[d]_{\pi_R} \ar[r]^{\tilde{F}} & \tilde{G}(A^{\circ}) \ar[rrdd]^{\tilde{\sigma}} \ar[d]^{\pi_{A^{\circ}}} \\ \Gamma \ar[rrrd]_{\rho} \ar[r]^{F} & G(A^{\circ}) \ar@{.>}[rrd]^{\sigma} \\ & & & H^{\circ} \\}
\end{equation}
form a commutative diagram.

To prove the Main Theorem, we will show that under appropriate assumptions, there exists a group homomorphism $\sigma \colon G(A^{\circ}) \to H^{\circ}$ making the full diagram (\ref{D:FI-1}) commute.
Observe that it follows from our definitions that
if such $\sigma$ exists, then $\rho \colon G(R)^+ \to GL_n  (K)$ coincides on $\Gamma$ with the composition $\sigma \circ T$, where $T \colon G(R)^+ \to G(A^{\circ})$ is the group homomorphism induced by the ring homomorphism $t \colon R \to A^{\circ}$, defined as the composition of $f \colon R \to A$
with the canonical projection $s \colon A \to A^{\circ}.$

In this section, we will show that such a homomorphism $\sigma$ exists
if $R$ is a semilocal ring (Proposition \ref{P:FI-3}). Later, in \S 6, we will prove
that $\sigma$ also exists
if $H^{\circ}$ is reductive or if char $K = 0$ and the unipotent radical $U$ of $H^{\circ}$ is commutative (and, more generally, if $H^{\circ}$ satisfies condition (Z) introduced below). Towards this end, we will establish in this section a weaker statement, viz. that in these cases there exists a homomorphism $\bar{\sigma} \colon G(A^{\circ}) \to \bar{H}$ such that $\bar{\sigma} \circ \pi_{A^{\circ}} = \nu \circ \tilde{\sigma},$ where $Z(H^{\circ})$ is the center of $H^{\circ},$ $\bar{H} = H^{\circ}/ Z(H^{\circ})$, and $\nu \colon H^{\circ} \to \bar{H}$ is the canonical map (cf. Proposition \ref{P:FI-5}). Throughout this section, we will use the following result of Matsumoto:
\begin{lemma}\label{L:FI-100}{\rm (cf. \cite{M}, Corollary 2)}
Let $\Phi$ be a reduced irreducible root system of rank $\geq 2$ and let $G$ be the corresponding universal Chevalley-Demazure group scheme. If $S$ is a semilocal commutative ring, then $G(S) = G(S)^+.$
\end{lemma}

\noindent In particular, we see that $G(A^{\circ}) = G(A^{\circ})^+$ since $A^{\circ}$ is semilocal by Lemma \ref{L:AR-1}, so the canonical homomorphism $\pi_{A^{\circ}} \colon \tilde{G}(A^{\circ}) \to G(A^{\circ})$ is surjective. Consequently, the existence of $\sigma$ is equivalent to the triviality of the restriction $\tilde{\sigma} \colon K_2 (\Phi, A^{\circ}) \to H^{\circ}.$

\begin{prop}\label{P:FI-3}
Suppose $R$ is a commutative semilocal ring. Then there exists a group homomorphism $\sigma \colon G(A^{\circ}) \to H^{\circ}$ making the diagram {\rm (\ref{D:FI-1})} commute. Moreover, if char $K = 0$, then $B:= A^{\circ}$ is a finite-dimensional $K$-algebra, and $\sigma$ can be viewed as a homomorphism $G(B) \to H^{\circ}$ such that the composition $\sigma \circ T$, where $T \colon G(R)^+ \to G(B)$ is induced by $t \colon R \to B$, coincides with $\rho$ on a subgroup $\Gamma \subset G(R)^+$ of finite index.
\end{prop}
\begin{proof}
Observe that by Corollary \ref{C:AR-2}, $t(R^{\times})$ is Zariski-dense in $(A^{\circ})^{\times},$ where as above, $t \colon R \to A^{\circ}$ is the composition of the homomorphism $f \colon R \to A$ with the projection $s \colon A \to A^{\circ}.$
Let $\Delta = R^{\times} \cap f^{-1} (A^{\circ} \times \{ 1 \}).$ Then $\Delta$ has finite index in $R^{\times},$ so since $(A^{\circ})^{\times}$ is irreducible, we obtain that $f(\Delta)$ is Zariski-dense in $(A^{\circ})^{\times} \times \{1 \}.$ Fix a long root $\alpha$, and for $u, v \in A^{\times}$, let $\{u, v \}_{\alpha}$ denote the corresponding Steinberg symbol. Clearly
$$
\tilde{\tau} ( \{u, v \}_{\alpha}) = H_{\alpha} (uv) H_{\alpha} (u)^{-1} H_{\alpha} (v)^{-1},
$$
where for $r \in A^{\times}$, we set
$$
H_{\alpha} (r) = {{W}}_{\alpha} (r) {{W}}_{\alpha} (-1) \ \ \ {\rm and} \ \ \ {{W}}_{\alpha} (r) = \psi_{\alpha} (r) \psi_{-\alpha} (-r^{-1}) \psi_{\alpha} (r).
$$
Now by Proposition \ref{P:AR-1}, the map $A^{\times} \to A^{\times}, t \mapsto t^{-1}$ is regular, hence the map
$\Theta \colon A^{\times} \times A^{\times} \to H,$ $(u,v ) \mapsto \tilde{\tau} (\{ u, v \}_{\alpha})$ is also regular. On the other hand, as we noted earlier,
for $a, b \in R^{\times},$ we have $h_{\alpha} (ab) = h_{\alpha} (a) h_{\alpha} (b)$ (see \cite{Stb1}, Lemma 28).
So, by Proposition \ref{P:StG-1},
$$
\tilde{\tau} (\{ f(a), f(b) \} ) = \rho (h_{\alpha} (ab) h_{\alpha} (a)^{-1} h_{\alpha} (b)^{-1}) = 1
$$
for all $a, b \in R^{\times}.$
Since the closure of $f(R^{\times})$ in $A^{\times}$ contains $(A^{\circ})^{\times} \times \{1 \},$ it follows that $\tilde{\tau}$ vanishes on $((A^{\circ})^{\times} \times \{ 1 \}) \times ((A^{\circ})^{\times} \times \{ 1 \} ).$ But according to Corollary \ref{C:StG-1}, $\ker \pi_{A^{\circ}}$ is generated by the Steinberg symbols $\{ u, v \}_{\alpha}$ for any fixed long root $\alpha \in \Phi.$ Thus, $\tilde{\sigma}$ vanishes on $\ker \pi_{A^{\circ}}$, implying that the required homomorphism $\sigma \colon G(A^{\circ}) \to H^{\circ}$ exists. The fact that $B := A^{\circ}$ is a finite-dimensional $K$-algebra if char $K = 0$ has already been established in Proposition \ref{P:AR-2}; the remaining assertions follow.
\end{proof}

%If moreover $K$ is a field of characteristic 0, then since $A^{\circ}$ is a finite-dimensional $K$-algebra by Proposition \ref{P:AR-2}, the above proposition, combined with our earlier remark, yields
%\begin{prop}\label{P:FI-4}
%Suppose $R$ is a semilocal ring and char $K = 0.$ Then there exists a finite-dimensional algebra $B$, a ring homomorphism $t \colon R \to B$ with Zariski-dense image, and a group homomorphism $\sigma \colon G(B) \to H^{\circ}$ such that the composition $\sigma \circ T,$ where $T \colon G(R) \to G(B)$ is induced by $t$, coincides with $\rho$ on a subgroup $\Gamma \subset G(R)^+$ of finite index.
%\end{prop}

In particular, if $R = k$ is an infinite field of characteristic $\neq 2$ or $3$, then $R$ is automatically semilocal, $(\Phi, R)$ is a nice pair,
and $G(R)^+$ does not have any subgroups of finite index (since it contains no proper noncentral normal subgroups --- cf. \cite{T1}),
so Proposition \ref{P:FI-3} proves the existence of the homomorphism $\sigma$ in (BT). Its rationality will be established in \S 6.

To consider the cases when $H^{\circ}$ is reductive or when the unipotent radical $U$ of $H^{\circ}$ is commutative, we will need some additional information about the structure of $H^{\circ}.$ We begin with
\begin{prop}\label{P:FI-1}
The group $H^{\circ}$ coincides with $\tilde{\sigma}(\tilde{G}(A^{\circ}))$ and is its own commutator.
\end{prop}
\begin{proof}
Let $\psi_{\alpha} \colon A \to H$, $\alpha \in \Phi$, be the regular map constructed in Theorem \ref{T:ARR-1}. Then it follows from Proposition \ref{P:StG-1} that $\tilde{\sigma}(\tilde{G} (A^{\circ}))$ coincides with the (abstract) subgroup $\mathcal{H} \subset H$ generated by all the $\psi_{\alpha} (A^{\circ})$, with $\alpha \in \Phi.$ Since $\psi_{\alpha} (A^{\circ})$ is clearly a connected subgroup of $H$,
we have by (\cite{Bo}, Proposition 2.2) that
$\mathcal{H}$ is Zariski-closed and connected, hence $\mathcal{H} \subset H^{\circ}.$
Now by Corollary \ref{C:StG-2},
$[ \tilde{G}(A) : \tilde{G}(A^{\circ})] < \infty$,  from which it follows that $\tilde{\sigma} (\tilde{G}(A))$ is Zariski-closed. On the other hand, $\tilde{\sigma} (\tilde{G}(A))$ contains $\rho (G(R)^+)$, and therefore is Zariski-dense in $H$. Thus, $\tilde{\sigma} (\tilde{G}(A)) = H.$ So, $\mathcal{H}$ is a closed subgroup of finite index in $H$, hence $\mathcal{H} \supset H^{\circ}$ and consequently $\mathcal{H} = H^{\circ},$ proving our first claim. For the second claim, recall that $A^{\circ}$ is semilocal by Lemma \ref{L:AR-1}. Moreover, $A^{\circ}$ has no residue field isomorphic to $\mathbb{F}_2.$ Indeed, first observe that $A^{\circ}$ is a connected commutative algebraic with identity --- this is clear if char $K = 0$ and follows from our assumption that $R$ is noetherian and Lemmas \ref{L:AR-4} and \ref{L:AR-3} if char $K > 0.$ Hence by Lemma \ref{L:AR-2}, every abstract ideal $I \subset A^{\circ}$ is Zariski-closed, and therefore, the canonical map $A \to A/ I$ is a morphism of algebraic rings (\cite{Bo}, Theorem 6.8). Thus, the fact that $A^{\circ}$ is connected implies that it has no finite quotient rings, in particular no residue field isomorphic to $\mathbb{F}_2.$
So, since rank $\Phi \geq 2,$ it follows from (\cite{St1}, Corollary 4.4), that $\tilde{G}(A^{\circ})$ coincides with its commutator subgroup. Hence the same is true for $H^{\circ} = \tilde{\sigma} (\tilde{G}(A^{\circ})).$
\end{proof}

\begin{cor}\label{C:FI-5}
If $H^{\circ}$ has a Levi decomposition\footnotemark \ $H^{\circ} = U \rtimes S$, where $U$ is the unipotent radical of $H^{\circ}$ and $S$ is reductive, then $S$ is automatically semisimple. In particular, if $H^{\circ}$ is reductive, then it is semisimple and hence the center $Z(H^{\circ})$ is finite. \footnotetext{Recall that $H^{\circ}$ always has a Levi decomposition if char $K = 0$ (cf. \cite{Mos})}

\end{cor}
\begin{proof}
Since $H^{\circ} = [H^{\circ}, H^{\circ}]$, we have $S= [S, S],$ so $S$ is semisimple (\cite{Bo}, Corollary 14.2). In particular, if $H^{\circ}$ is reductive, then $H^{\circ} = S$ is semisimple, hence $Z(H^{\circ})$ is finite.
\end{proof}

As in the corollary, let $U$ be the unipotent radical of $H^{\circ}$ and $Z(H^{\circ})$ be the center of $H^{\circ}$. To give uniform statements of some results in \S 6, we introduce the following condition on $H^{\circ}$:

\vskip3mm

\noindent (Z)\ \  \parbox[t]{15cm}{$Z(H^{\circ}) \cap U = \{ e \}.$}

\vskip3mm

\begin{prop}\label{P:FI-2} \ \
\newline {\rm (i)}\parbox[t]{16cm}{Suppose $H^{\circ}$ satisfies {\rm (Z)}. Then $Z(H^{\circ})$ is finite. Moreover, if char $K = 0$, then $Z(H^{\circ})$ is contained in any Levi subgroup of $H^{\circ}$.}

\vskip1mm

\noindent {\rm (ii)}\parbox[t]{15.5cm}{Assume that char $K = 0$ and that $U$ is commutative. Then $H^{\circ}$ satisfies {\rm (Z)}.}
\end{prop}
\begin{proof}
(i) The quotient $H^{\circ}/U$ is a reductive algebraic group that coincides with its commutator, so $Z(H^{\circ}/ U)$ is finite by (\cite{Bo}, Corollary 14.2). On the other hand, by (Z), the restriction to $Z(H^{\circ})$ of the canonical map $H^{\circ} \to H^{\circ}/U$ is injective, from which the finiteness of $Z(H^{\circ})$ follows. Now suppose char $K = 0$, and let $S$ be any Levi subgroup so that $H^{\circ} = U \rtimes S$. Since $Z(H^{\circ})$ is a finite abelian group, it is reductive, and therefore
a suitable conjugate of it is contained in $S$ (cf. \cite{Mos}). So, being central, $Z(H^{\circ})$ itself is contained in $S$.

%$Z(H^{\circ})$
%Using the Levi decomposition over $K$ (cf. \cite{Mos}), we can write $H^{\circ} = U \rtimes S$, with $S$ reductive. Since $H^{\circ} = [H^{\circ}, H^{\circ}],$ we conclude that $S = [S, S],$ so $S$ is semisimple (cf. \cite{H2}, Lemma 19.5).

\vskip1mm

(ii) Let $H^{\circ} = U \rtimes S$ be a Levi decomposition of $H^{\circ}.$ It follows from Remark 7.3 in \cite{Bo} that since char $K = 0$, we have $U \simeq (K^m, +)$ where $m = \dim U,$ and then the action of $S$ on $U$ yields a rational representation of $S$ on $K^m.$ To simplify notation, we will identify $U$ with $K^m$ for the rest of the proof. By Weyl's Theorem , the representation of $S$ on $U$ is completely reducible (see \cite{H1}, Theorem 6.3 and \cite{H2}, Theorem 13.2), and the fact that $H^{\circ} = [H^{\circ}, H^{\circ}]$ implies that it cannot contain the trivial representation. Thus, $U$ has no nonzero vectors fixed by the action of $S$, and consequently $Z(H^{\circ}) \cap U = \{ e \}.$
%i.e. we have $U = \oplus U_i,$ where the $U_i$ are irreducible representations (see \cite{H1}, Theorem 6.3 and \cite{H2}, Theorem 13.2). Now the fact that $H^{\circ} = [H^{\circ}, H^{\circ}]$ implies that this decomposition cannot contain the trivial representation, so $U$ has no nonzero vectors fixed by the action of $S$. On the other hand, if $(u,s) \in Z(H^{\circ})$, then clearly $u$ must be fixed by $S$, hence $u = 0$ and $Z(H^{\circ}) \subset S.$ Thus, $Z(H^{\circ}) \cap U = \{ e \},$ as claimed.
%Then in fact $Z(H^{\circ}) \subset Z(S)$, so in particular, $Z(H^{\circ})$ is finite (cf. \cite{Bo}, Corollary 14.2).
\end{proof}

Now set $\bar{H} = H^{\circ} / Z(H^{\circ})$.
%It follows from Proposition \ref{P:FI-2} that $\bar{H} = U \rtimes \bar{S}$, where $\bar{S} = S/ Z(H^{\circ})$. Moreover,
Since $Z(H^{\circ})$ is a closed normal subgroup of $H^{\circ}$, by (\cite{Bo}, Theorem 6.8), $\bar{H}$ is an algebraic group and the canonical map $\nu \colon H^{\circ} \to \bar{H}$ is a morphism of algebraic groups. We let $\bar{\rho} = \nu \circ \rho.$ Since $H^{\circ} = \tilde{\sigma}(\tilde{G}(A^{\circ}))$ by Proposition \ref{P:FI-1}, and $K_2 (\Phi, A^{\circ}) = \ker \pi_{A^{\circ}}$ is a central subgroup of $\tilde{G}(A^{\circ})$ by Corollary \ref{C:StG-1}, it is clear that $\tilde{\sigma}$ descends to a homomorphism $\bar{\sigma} \colon G(A^{\circ}) \to \bar{H}$ such that $\bar{\sigma} \circ \pi_{A^{\circ}} = \nu \circ \tilde{\sigma}.$ We will now use this observation, in conjunction with our previous results on algebraic rings, to prove the following:

\begin{prop}\label{P:FI-5}
If either char $K = 0$ or $H^{\circ}$ is reductive, then $B := A^{\circ}$ is a finite-dimensional $K$-algebra, and $\bar{\sigma}$ can be viewed as a homomorphism $G(B) \to \bar{H}$ such that the composition $\bar{\sigma} \circ T$, where $T \colon G(R) \to G(B)$ is induced by $t \colon R \to B$, coincides with $\bar{\rho}$ on a subgroup $\Gamma \subset G(R)^+$ of finite index.
\end{prop}
We begin with the following
\begin{lemma}\label{L:FI-1}
Let $\ell$ be the nilpotency class of $U$. If $J$ is the Jacobson radical of $A^{\circ},$ then $J^{\ell + 1} = \{0 \}.$ In particular, if $H^{\circ}$ is reductive, then $J = \{ 0 \}.$
\end{lemma}
\begin{proof}
Let $\bar{A} = A^{\circ}/J.$ We have the following commutative diagram
$$
\xymatrix{ \tilde{G}(A^{\circ}) \ar[d]_{\pi_{A^{\circ}}} \ar[r]^{\gamma} & \tilde{G}(\bar{A}) \ar[d]^{\pi_{\bar{A}}} \\ G(A^{\circ}) \ar[r]^{\delta} & G(\bar{A})}
$$
where $\gamma$ and $\delta$ are induced by the canonical map $A^{\circ} \to \bar{A}.$ It follows that for $M := \ker \gamma$, we have $\pi_{A^{\circ}} (M) \subset N := \ker \delta.$ On the other hand, there is an embedding $G \hookrightarrow GL_d$ as group schemes over $\Z$ (cf. \cite{Bo1}, 3.4), and then $N = G(A^{\circ}) \cap GL_d (A^{\circ}, J)$, where $GL_d (A^{\circ}, J)$ is the congruence subgroup modulo $J$. It is well-known, and easily seen by direct computation, that for $s,t \geq 1$ we have
$$
[GL_d (A^{\circ}, J^s), GL_d (A^{\circ}, J^t)] \subset GL_d (A^{\circ}, J^{s+t}).
$$
Since $A^{\circ}$ is artinian by Lemma \ref{L:AR-2}, $J$ is nilpotent (see \cite{At}, Proposition 8.4), implying that $GL_d (A^{\circ}, J)$ is a nilpotent group. But $\ker \pi_{A^{\circ}}$ is central in $\tilde{G}(A^{\circ})$, so we conclude that $M$ is nilpotent.

On the other hand, $M$ coincides with the normal subgroup of $\tilde{G}(A^{\circ})$ generated by the $\tilde{x}_{\alpha}(a)$, with $a \in J$ and $\alpha \in \Phi.$ Since $J$ is connected by Lemma \ref{L:AR-2}, we see that $\tilde{\sigma}(M)$ is Zariski-closed, connected, and nilpotent. These facts imply that $\tilde{\sigma}(M)$ is contained in the radical of $H^{\circ}$; but since $H^{\circ} = [H^{\circ}, H^{\circ}],$ the radical coincides with the unipotent radical, and therefore $\tilde{\sigma} (M) \subset U.$ Now if we denote by $\mathcal{C}^i T$ the $i$-th term of the lower central series of a group $T$, then it follows from the Steinberg commutator relations that there is a root $\alpha \in \Phi$ such that $\tilde{x}_{\alpha} (J^i) \subset \mathcal{C}^{i-1} M$ for each $i \geq 1.$
Indeed, if all roots of $\Phi$ are of the same length, then the Steinberg commutator relations have the form
$$
[\tilde{x}_{\alpha} (s), \tilde{x}_{\beta} (t)] = \tilde{x}_{\alpha + \beta} (\varepsilon st),
$$
with $\varepsilon = \pm 1$. Since the Weyl group acts transitively on roots of the same length, it is clear that we have the required inclusion for any root $\alpha.$
Furthermore, if $\Phi$ is of type $G_2$, then the long roots form a subsystem of type $A_2,$ so the inclusion holds in this case as well.
Finally, suppose that $\Phi$ is of type $B_2.$ Keeping the same notations as in \S \ref{S:ARR}, we have
\begin{equation}\label{E:FI-1}
[\tilde{x}_{\varepsilon_1} (s), \tilde{x}_{\varepsilon_2}(t)] = \tilde{x}_{\varepsilon_1 + \varepsilon_2} (2st).
\end{equation}
Since $2 \in (A^{\circ})^{\times}$, we obtain that $\tilde{x}_{\varepsilon_1 + \varepsilon_2} (J^2) \subset [M, M]= \mathcal{C}^1 M.$ Similarly, we have
$$
[\tilde{x}_{\varepsilon_1 + \varepsilon_2} (t), \tilde{x}_{-\varepsilon_2}(s)] = \tilde{x}_{\varepsilon_1} (ts) \tilde{x}_{\varepsilon_1 - \varepsilon_2} (-ts^2),
$$
so that
\begin{equation}\label{E:FI-2}
[\tilde{x}_{\varepsilon_1 + \varepsilon_2} (t), \tilde{x}_{-\varepsilon_2}(s)] ([\tilde{x}_{\varepsilon_1 + \varepsilon_2} (t), \tilde{x}_{-\varepsilon_2}(-s)])^{-1} = \tilde{x}_{\varepsilon_1} (2 ts).
\end{equation}
Hence $\tilde{x}_{\varepsilon_1} (J^2) \subset [M,M] = \mathcal{C}^1 M.$ Now suppose by induction that $\tilde{x}_{\varepsilon_1 + \varepsilon_2} (J^i), \tilde{x}_{\varepsilon_1} (J^i) \subset \mathcal{C}^{i-1} M.$ Let $s \in J^i, t \in J.$ Then (\ref{E:FI-1}) shows that $\tilde{x}_{\varepsilon_1 + \varepsilon_2} (2st) \in [\mathcal{C}^{i-1} M, M] = \mathcal{C}^i M,$ hence $\tilde{x}_{\varepsilon_1 + \varepsilon_2} (J^{i+1}) \subset \mathcal{C}^i M.$ Similarly, (\ref{E:FI-2}) gives us that $\tilde{x}_{\varepsilon_1} (J^{i+1}) \subset \mathcal{C}^i M$, as required.

So, if $\mathcal{C}^{\ell} U = \{1 \},$ we have $\psi_{\alpha} (J^{\ell + 1}) = \{ 1 \}$
(cf. Proposition \ref{P:StG-1}).
But the maps $\psi_{\alpha}$ are injective (cf. Theorem \ref{T:ARR-1}), so we obtain that $J^{\ell + 1} = \{ 0 \}$, as claimed.
\end{proof}

\noindent {\it Proof of Proposition \ref{P:FI-5}.} \ The fact that $B:= A^{\circ}$ is a finite-dimensional $K$-algebra if char $K = 0$ has already been established in Proposition \ref{P:AR-2}. Now suppose that $H^{\circ}$ is reductive. Then, according to the lemma, we have $J = \{ 0 \}$, and $B$ is a finite-dimensional $K$-algebra by Proposition \ref{P:AR-20}. The remaining assertions follow.

\vskip5mm

\section{Rationality}\label{S:R}

In this section, we will complete the proof of the Main Theorem (cf. Theorem \ref{T:R-1}) by first showing that the abstract group homomorphisms $\sigma \colon G(B) \to H^{\circ}$ and $\bar{\sigma} \colon G(B) \to \bar{H}$ constructed in Propositions \ref{P:FI-3} and \ref{P:FI-5}, respectively, are in fact morphisms of algebraic groups, and then by lifting $\bar{\sigma}$ in the latter case to a morphism of algebraic groups $\sigma \colon G(B) \to H^{\circ}$ that makes the diagram (\ref{D:FI-1}) commutative.

Note that in the statements concerning the rationality of $\sigma$ and $\bar{\sigma}$, we are implicitly using the fact that the functor of restriction of scalars $R_{B/K}$ enables us to endow $G(B)$ with a natural structure of a connected algebraic group over $K$. We refer to (\cite{DG}, Chapter I, \S 1, 6.6) and (\cite{Oes}, Appendices 2 and 3) for a general discussion of restriction of scalars. In the present context, all we need is that since $B$ is a finite-dimensional $K$-algebra, $R_{B/K}(G)$ (or more precisely $R_{B/K}(_{B}G)$, where $_{B}G$ is obtained from $G$ by the base change $\Z \hookrightarrow B$) is an algebraic $K$-group such that $R_{B/K} (G)(K)$ can be naturally identified with $G(B)$, yielding a structure of an algebraic $K$-group on the latter. Also note that for each root $\alpha \in \Phi$, we have a morphism $R_{B/K} (e_{\alpha}) \colon R_{B/K} (\mathbb{G}_a) \to R_{B/K} (G)$. Now, $R_{B/K} (\mathbb{G}_a) (K) \simeq B$ is an irreducible $K$-variety. On the other hand, since $B$ is a finite-dimensional $K$-algebra, it follows from Lemma \ref{L:FI-100} that $G(B) = G(B)^+,$ i.e. the images $R_{B/K} (e_{\alpha}) (R_{B/K} (\mathbb{G}_a)(K))$ generate $R_{B/K} (G)(K)$. Using (\cite{Bo}, Proposition 2.2), we conclude that $G(B)$ is a connected algebraic group over $K$.

We now need to introduce some additional notations that will be used later.
As above, let $\Phi$ be a reduced irreducible root system of rank $\geq 2$, and let $G$ be the universal Chevalley-Demazure group scheme of type $\Phi.$ For the remainder of this section, we fix an ordering on $\Phi,$ and
let $\Phi^+$ and $\Phi^-$ denote the corresponding subsystems of positive and negative roots, respectively; also, let $\Pi \subset \Phi^+$ be a system of simple roots. Set $m = \vert \Phi^+ \vert = \vert \Phi^- \vert,$ and $\ell = \vert \Pi \vert.$ Next, let $U^+$ and $U^-$ be the standard subschemes associated to $\Phi^+$ and $\Phi^-$, respectively, and let $T$ be the usual maximal torus (cf. \cite{Chev}, \S 4).

Now suppose $R$ is a commutative ring
such that $(\Phi, R)$ is a nice pair and that $K$ is an algebraically closed field. Throughout this section, we will assume that $R$ is noetherian if char $K > 0.$ Given a representation $\rho \colon G(R)^+ \to GL_n (K)$, let $A$ be the algebraic ring associated to $\rho$ (Theorem \ref{T:ARR-1}). In this section,
we will assume that $B := A^{\circ}$ is a finite-dimensional $K$-algebra, so that $G(A^{\circ})= G(B)$ has a natural structure of a connected algebraic $K$-group, as explained above. We note that $B$ is indeed a finite-dimensional $K$-algebra if either char $K = 0$ or $H^{\circ}$ is reductive (see Proposition \ref{P:FI-5}), which are precisely the cases needed to complete the proof of the Main Theorem.
As a matter of convention,
whenever we need to emphasize that we are working with the maps involving $A^{\circ}$ constructed in \S \ref{S:FI}, we will use the notation $G(A^{\circ})$ rather than $G(B).$

We begin with the following lemma, which describes the ``big cell" of $G(B).$

\begin{lemma}\label{L:R-1}
The product map $p: U^- \times T \times U \to G$ gives an isomorphism onto an open subscheme $\Omega \subset G$ over $\Z$. Consequently, if $B$ is a finite-dimensional $K$-algebra, then $\Omega (B)$ is a Zariski-open subvariety of $G(B).$
\end{lemma}
\begin{proof}
By (\cite{Bo1}, Lemma 4.5), $p$ is an isomorphism onto a principal open subscheme $\Omega \subset G$ defined by some $d \in \Z[G]$ (which is in fact
a matrix coefficient of the $m$-th exterior power of the adjoint representation). It follows that
$$
\Omega (B) = \{ g \in G(B) \mid d(g) \in B^{\times} \}.
$$
Since $d \colon G(B) \to B$ (or, more precisely, $R_{B/K} (d)$) is a regular map of algebraic $K$-varieties, and $B^{\times} \subset B$ is Zariski-open by Proposition \ref{P:AR-1}, we obtain that $\Omega (B)$ is an open subvariety of $G(B)$, as claimed.
\end{proof}

As a first step in proving that the homomorphisms $\sigma$ and $\bar{\sigma}$ are regular, we will now show that their restrictions to $\Omega(A^{\circ})$ are regular.

\begin{lemma}\label{L:R-2}
Let $\sigma \colon G(A^{\circ}) \to H^{\circ}$ and $\bar{\sigma} \colon G(A^{\circ}) \to \bar{H}$ be the group homomorphisms constructed in Propositions {\ref{P:FI-3}} and {\ref{P:FI-5}}, respectively. Then their restrictions $\sigma \vert_{\Omega (A^{\circ})}$ and $\bar{\sigma} \vert_{\Omega (A^{\circ})}$ to $\Omega(A^{\circ})$ are regular.
\end{lemma}
\begin{proof}
Recall that there exist isomorphisms of group schemes over $\Z$
$$
\omega^+ \colon (\mathbb{G}_a)^m \to U, \ \ \ \omega^- \colon (\mathbb{G}_a)^m \to U^-, \ \ \ {\rm and} \ \ \ \omega \colon (\mathbb{G}_m)^{\ell} \to T
$$
such that for any commutative ring $R$, we have
$$
\omega^+ ((u_{\alpha})_{\alpha \in \Phi^+}) = \prod_{\alpha \in \Phi^+} e_{\alpha} (u_{\alpha}), \ \ \ \omega^- ((u_{\alpha})_{\alpha \in \Phi^-}) = \prod_{\alpha \in \Phi^-} e_{\alpha} (u_{\alpha}), \ \ \ {\rm and} \ \ \ \omega ((t_{\alpha})_{\alpha \in \Pi}) = \prod_{\alpha \in \Pi} h_{\alpha} (t_{\alpha}),
$$
where the roots in $\Phi^+$ and $\Phi^-$ are listed in an arbitrary (but {\it fixed}) order (cf. \cite{Chev}, \S 4). Let
$$
\theta \colon (\G_a)^m \times (\G_m)^{\ell} \times (\G_a)^m \to G
$$
be the composition of the isomorphism
$$
\omega^- \times \omega \times \omega^+ \colon (\G_a)^m \times (\G_m)^{\ell} \times (\G_a)^m \to U^- \times T \times U
$$
with the product map $p \colon U^- \times T \times U \to G.$ Then it follows from Lemma \ref{L:R-1} that $\theta$ is an isomorphism onto the Zariski-open subscheme $\Omega \subset G$, and consequently $\theta$ induces an isomorphism of $K$-varieties
$$
\theta_{A^{\circ}} \colon (A^{\circ})^m \times ((A^{\circ})^{\times})^{\ell} \times (A^{\circ})^m \to \Omega (A^{\circ}).
$$
Next, let $\psi_{\alpha} \colon A^{\circ} \to H$, $\alpha \in \Phi$, be the regular maps constructed in Theorem \ref{T:ARR-1}, and $H_{\alpha} \colon (A^{\circ})^{\times} \to H$, $\alpha \in \Phi$, be the regular maps introduced in the proof of Proposition \ref{P:FI-3}.
%with $\alpha \in \Phi$, be the regular map constructed in Theorem \ref{T:ARR-1}, and $\chi_{\alpha} \colon (A^{\circ})^{\times} \to H$ be the regular map introduced in the proof of Proposition \ref{P:FI-3}.
Define the regular map
$$
\kappa \colon (A^{\circ})^m \times ((A^{\circ})^{\times})^{\ell} \times (A^{\circ})^m \to H
$$
by
$$
\kappa ((u_{\alpha})_{\alpha \in \Phi^-}, (t_{\alpha})_{\alpha \in \Pi}, (u_{\alpha})_{\alpha \in \Phi^+}) = \left( \prod_{\alpha \in \Phi^-} \psi_{\alpha} (u_{\alpha}) \right) \cdot \left( \prod_{\alpha \in \Pi} H_{\alpha} (t_{\alpha}) \right) \cdot \left( \prod_{\alpha \in \Phi^+} \psi_{\alpha} (u_{\alpha}) \right).
$$
By our construction, $\sigma (e_{\alpha} (a)) = \psi_{\alpha} (a)$ for any $a \in A^{\circ}$ and all $\alpha \in \Phi,$ and consequently $(\sigma (h_{\alpha} (a)) = H_{\alpha} (a)$ for all $a \in (A^{\circ})^{\times}$ and $\alpha \in \Phi.$ Thus we have the following commutative diagram
$$
\xymatrix{ & (A^{\circ})^m \times ((A^{\circ})^{\times})^{\ell} \times (A^{\circ})^m \ar[ld]_{\theta} \ar[rd]^{\kappa} \\ G(A^{\circ}) \supset \Omega(A^{\circ})  \ar[rr]^{\sigma} & & H}
$$
Since $\kappa$ is regular and $\theta$ is an isomorphism of algebraic varieties over $K$, we conclude that $\sigma \vert_{\Omega (A^{\circ})}$ is regular. A similar argument shows that $\bar{\sigma} \vert_{\Omega (A^{\circ})}$ is regular.

\end{proof}

We can now prove
\begin{prop}\label{P:R-1}
The homomorphisms $\sigma \colon G(B) \to H^{\circ}$ and $\bar{\sigma} \colon G(B) \to \bar{H}$ constructed in Propositions \ref{P:FI-3} and \ref{P:FI-5}, respectively, are morphisms of algebraic groups.
\end{prop}
Indeed, as we remarked above, under the hypotheses of Propositions \ref{P:FI-3} and \ref{P:FI-5}, $G(B)$ is a connected algebraic group over $K$. Thus, the proposition follows immediately from
Lemma \ref{L:R-2} and the following (elementary)
\begin{lemma}\label{L:R-3}
Let $K$ be an algebraically closed field and let $\mathscr{G}$ and $\mathscr{G}'$ be affine algebraic groups over $K$, with $\mathscr{G}$ connected. Suppose $f \colon \mathscr{G} \to \mathscr{G}'$ is an abstract group homomorphism\footnotemark and assume there exists a Zariski-open set $V \subset \mathscr{G}$
%containing the identity element $e_G$
such that $\varphi := f \vert_V$ is a regular map. Then $f$ is a morphism of algebraic groups. \footnotetext{Here we tacitly identify $\mathscr{G}$ and $\mathscr{G}'$ with the corresponding groups $\mathscr{G}(K)$ and $\mathscr{G}'(K)$ of $K$-points.}
%there exists a Zariski-open set Let $G$ be a connected affine algebraic group, and let $f \colon G \to GL_n (K)$ be a group homomorphism, with $K$ an algebraically closed field.  Suppose there exists a Zariski- open set $V \subset G$ containing the identity element 1 such that $\varphi = f \vert_V$ is a regular map. Then $f$ is a morphism of algebraic groups.
\end{lemma}
\begin{proof}
Clearly, since $f$ is a group homomorphism, given $x,y \in V$ such that $xy \in V$, we have $\varphi (xy) = \varphi(x) \varphi(y).$
Now fix any $a \in V$, and consider the regular map $\varphi_a \colon a V \to \mathscr{G}'$ given by $\varphi_a (a u) = \varphi(a) \varphi(u).$ If $a u = v \in a V \cap V$, then $\varphi_a (au) = \varphi (a) \varphi (u) = \varphi (au) = \varphi (v).$ Hence, $\varphi$ and $\varphi_a$ define the same rational map. So, since $\varphi_a$ is defined on $aV$, it follows that $\varphi$ is defined on $\cup_{ a \in V} a V = VV = \mathscr{G}$,  where the last equality follows from (\cite{Bo}, Proposition 1.3) as $\mathscr{G}$ is connected. Thus, $\varphi$ extends to a regular map on $\mathscr{G}$, which we will also denote by $\varphi$. Observe now that the map $\mathscr{G} \times \mathscr{G} \to \mathscr{G}'$, $(x,y) \mapsto \varphi(xy)^{-1} \varphi(x) \varphi(y)$ coincides with the constant map $(x,y) \mapsto e_{\mathscr{G}'}$ on $V \times V$. So, the density of $V$ implies that $\varphi(xy) = \varphi(x) \varphi(y)$ for all $x,y \in \mathscr{G}.$ Thus, $\varphi$ and $f$ are both group homomorphisms that coincide on $V$, hence $\varphi \equiv f$ as $V V = \mathscr{G}.$ So, $f$ is a morphism of algebraic groups, as claimed.
%is a regular map on $G$. It remains to show that $\varphi$ is a group homomorphism. Observe that the map
%$G \times G \to G',$ $(x,y) \mapsto \varphi(xy)^{-1} \varphi(x) \varphi(y)$ coincides with the trivial homomorphism $(x,y) \mapsto e_{G'}$
%on $V \times V$, and since $V$ is dense, $\varphi (xy) = \varphi(x) \varphi(y)$ for all $x, y \in G$. Thus, $\varphi$ and $f$ coincide on $V$. Since $VV = G,$ we see that $\varphi \equiv f$, and $f$ is a morphism of algebraic groups, as claimed.
\end{proof}

The next step is to show that under appropriate assumptions, the morphism of algebraic groups $\bar{\sigma} \colon G(B) \to \bar{H}$ can be lifted to a morphism $\sigma \colon G(B) \to H^{\circ}$ making the diagram (\ref{D:FI-1}) commutative. For this, we first establish some structural results for $G(B)$ as an algebraic $K$-group, where $B$ is an arbitrary finite-dimensional $K$-algebra.
Let $J = J(B)$ be the Jacobson radical of $B$. By the Wedderburn-Malcev Theorem (see \cite{P}, Theorem 11.6), there exists a semisimple subalgebra $\bar{B} \subset B$ such that $B = \bar{B} \oplus J.$ Then $\bar{B} \simeq B/J \simeq K \times \cdots \times K$ ($r$ copies). Let $e_i = (0, \dots, 0, 1, 0, \dots, 0) \in \bar{B}$ be the $i$th standard basis vector. Then we have $B = \oplus_{i=1}^r B_i,$ where $B_i = e_i B.$ Clearly, $B_i = \bar{B}_i \oplus J_i$ with $\bar{B}_i = e_i \bar{B} \simeq K$ and $J_i = e_i J$, so in particular, $B_i$ is a local $K$-algebra with maximal ideal $J_i.$ For an ideal $\mathfrak{b} \subset B$, we let $G(B, \mathfrak{b})$ denote the congruence subgroup modulo $\mathfrak{b},$ i.e. the kernel of the natural morphism of algebraic $K$-groups $G(B) \to G(B/ \mathfrak{b})$; it is clear that $G(B, \mathfrak{b})$ is a closed normal subgroup of $G(B).$
\begin{prop}\label{P:R-3}
{\rm (i)} $G(B) = G(B,J) \rtimes G(\bar{B})$ is a Levi decomposition of $G(B)$;

\vskip2mm

\noindent {\rm (ii)} $G(B) \simeq G(B_1) \times \cdots \times G(B_r)$, where each $B_i$ is a finite-dimensional local $K$-algebra;

\vskip2mm

\noindent {\rm (iii)} \parbox[t]{15.6cm}{Suppose $B = K \oplus J$ is a finite-dimensional commutative local $K$-algebra. Let $d \geq 1$ be such that $J^d = \{ 0 \}$, and for each $k = 1, \dots, d-1,$ let $s_k = \dim_K J^k/ J^{k+1}.$ Then, for $k = 1, \dots, d-1,$ the quotient $\mathcal{G}_{k} := G(B, J^{k})/ G(B, J^{k +1})$ is isomorphic as an algebraic $K$-group
to $\mathfrak{g} \times \cdots \times \mathfrak{g}$ ($s_{k}$ copies), where $\mathfrak{g}$ is the Lie algebra of $G$, considered as an algebraic group in terms of the underlying vector space. Furthermore, the conjugation action of $G(K)$ on $\mathcal{G}_{k}$ is the sum of $s_{k}$ copies of the adjoint representation.}
%If $B = K \oplus J$, where $J = J(B)$, is a finite-dimensional commutative local $K$-algebra, and $d \geq 1$ is such that $J^{d} = \{ 0 \}$, then for $k = 1, \dots, d-1,$ the quotient $\mathcal{G}_{k} := G(B, J^{k})/ G(B, J^{k +1})$ is isomorphic to $\mathfrak{g} \times \cdots \times \mathfrak{g}$ ($s_{k}$ copies) as algebraic $K$-groups, where $\mathfrak{g}$ is the Lie algebra of $G$, so that the conjugation action of $G(K)$ on $\mathcal{G}_{k}$ becomes the sum of $s_{k}$ copies of the adjoint representation.

\end{prop}
\begin{proof}
(i) Fix an embedding $G \hookrightarrow GL_n$ as group schemes over $\Z$ (cf. \cite{Bo1}, 3.4). Then $G(B, J) = G(B) \cap GL_n (B, J)$. Note that the embedding $\bar{B} \hookrightarrow B$ induces a section for the canonical morphism $G(B) \to G(B/J)$, whence the semi-direct product decomposition \begin{equation}\label{E:R-1}
G(B) = G(B, J) \rtimes G(\bar{B}).
\end{equation}
Since $\bar{B} = K \times \cdots \times K,$ the group $G(\bar{B}) = G(K) \times \cdots \times G(B)$ is semisimple, in particular reductive (see \cite{Stb1}, Theorem 6). On the other hand, it is well-known and easy to check that for any $a, b \geq 1$, we have
$$
[GL_n (B, J^a), GL_n (B, J^b)] \subset GL_n (B, J^{a+ b}).
$$
Since $J^d = \{ 0 \}$ for some $d \geq 1$ (cf. \cite{At}, Proposition 8.4), we conclude that the group $GL_n (B, J)$ is nilpotent, and therefore $G(B, J)$ is also nilpotent.
Now since $G(B)$ is connected as an algebraic $K$-group, the decomposition (\ref{E:R-1}) implies the connectedness of $G(B,J).$ Thus, $G(B,J)$ is contained in the radical of $G(B).$
But $G(B)/ G(B, J) \simeq G(\bar{B})$ is reductive, so we see that
$G(B, J)$ is precisely the radical of $G(B).$ Finally, since $G(B)$ coincides with its commutator subgroup (see \cite{St1}, Corollary 4.4), the radical of $G(B)$ cannot contain any torus, hence coincides with the unipotent radical.

\vskip2mm

\noindent (ii) is obvious.

\vskip2mm

\noindent (iii) Let $C_{k} = B/ J^{k + 1}.$ It follows from (\cite{M}, Corollary 2)  that if $\bar{J} := J/J^{k+1}$ denotes the image of $J$ in $C_{k},$ then the canonical map $G(B, J^k) \to G(C_k, \bar{J}^k)$ is surjective, and therefore $\mathcal{G}_{k}$ can be identified with $G(C_{k}, \bar{J}^{k}).$ Now let $\{ v_1, \dots, v_{s_{k}} \}$ be a $K$-basis of $\bar{J}^{k} = J^{k}/ J^{k + 1}.$ Then an element $X \in GL_n (C_{k}, \bar{J}^{k})$ can be written in the form $X = I_n + X_1 v_1 + \cdots + X_{s_{k}} v_{s_{k}}$, for some $X_j \in M_n (K).$ Let $F$ be a regular function on $GL_n.$ Then
$$
F(X) = F(I_n) + (\nabla_{I_n} F \cdot X_1)v_1 + \cdots + (\nabla_{I_n} F \cdot X_k) v_k,
$$
where $\nabla_{I_n}$ is the gradient of $F$ evaluated at the identity $I_n.$
Thus, considering all the regular functions on $GL_n$ that vanish on $G$, we see that $X \in G(C_{k}, \bar{J}^{k})$ if and only if $X_i \in \mathfrak{g}$ for all $i = 1, \dots, s_{k},$ yielding an isomorphism
$$
\mathcal{G}_{k} \simeq G(C_{k}, \bar{J}^{k}) \simeq \mathfrak{g} \times \cdots \times \mathfrak{g}.
$$
Finally, it is clear that the conjugation action of $G(K)$ on $\mathcal{G}_{k}$ can be identified with the direct sum of copies of the adjoint representation, which completes the proof.
\end{proof}

Now we prove

\begin{prop}\label{P:R-2} In each of the following situations
\vskip1mm

\noindent {\rm (i)} $H^{\circ}$ is reductive (hence semisimple);

\vskip1mm

\noindent {\rm (ii)} char $K = 0$ and $H^{\circ}$ satisfies condition {\rm (Z)} introduced in \S 5

\vskip1mm
\noindent there exists a morphism of algebraic groups $\sigma \colon G(A^{\circ}) \to H^{\circ}$ such that $\tilde{\sigma} = \sigma \circ \pi_{A^{\circ}},$ i.e. the diagram {\rm (\ref{D:FI-1})} commutes.
\end{prop}
\begin{proof}
By Proposition \ref{P:FI-5}, in both cases, $B:= A^{\circ}$ is a finite-dimensional $K$-algebra, so $G(A^{\circ}) = G(B)$ has a natural structure of a connected algebraic $K$-group. Furthermore, in either case, the center $Z(H^{\circ})$ is finite (see Corollary \ref{C:FI-5} and Proposition \ref{P:FI-2}), so the canonical morphism $\nu \colon H^{\circ} \to \bar{H}$ is a central isogeny.

According to Proposition \ref{P:R-1}, $\bar{\sigma} \colon G(A^{\circ}) \to \bar{H}$ is a morphism of algebraic groups, which, by our construction, satisfies $\nu \circ \tilde{\sigma} = \bar{\sigma} \circ \pi_{A^{\circ}}$, in the notations introduced earlier. As we already noted in the proof of Proposition \ref{P:FI-5}, in case (i), we have $J(A^{\circ}) = \{ 0 \}$, so $A^{\circ} \simeq K \times \cdots \times K$ (Proposition \ref{P:AR-20}), and therefore $G(A^{\circ}) = G(K) \times \cdots \times G(K)$ is a semisimple simply connected algebraic group. Then, according to (\cite{BT1}, Proposition 2.24), there exists a morphism of algebraic groups $\sigma \colon G(A^{\circ}) \to H$ such that
$$
\nu \circ \sigma = \bar{\sigma}.
$$
We will next show that such a $\sigma$ also exists in case (ii). Pick a semisimple subalgebra $\bar{B} \subset B:=A^{\circ}$ such that $B = \bar{B} \oplus J$; then by
Proposition \ref{P:R-3},
$$
G(B) = G(B, J) \rtimes G(\bar{B})
$$
is a Levi decomposition of $G(B) = G(A^{\circ})$. Also, since $\bar{B} \simeq K \times \cdots \times K$, the group $G(\bar{B}) = G(K) \times \cdots \times G(K)$ is semisimple and simply connected. Set
$$
\bar{U} = \bar{\sigma} (G(A^{\circ}, J)) \ \ \  {\rm and} \ \ \  \bar{S} = \bar{\sigma} (G(\bar{B})).
$$
Then
$$
\bar{H} = \bar{U} \rtimes \bar{S}
$$
is a Levi decomposition of $\bar{H}$. Furthermore, setting $S = (\nu^{-1} (\bar{S}))^{\circ}$, we have that
$$
H^{\circ} = U \rtimes S,
$$
where $U$ is the unipotent radical of $H^{\circ}$, is a Levi decomposition of $H^{\circ}$. According to Proposition \ref{P:FI-5}, $Z(H^{\circ}) \subset S$, so that $\bar{S} = S/ Z(H^{\circ})$ and the restriction $\nu_U = \nu \vert_U \colon U \to \bar{U}$ is an isomorphism.
Since $G(\bar{B})$ is simply connected, there exists a morphism of algebraic groups $\sigma_S \colon G(\bar{B}) \to S$ such that
$$
\nu \circ \sigma_S = \bar{\sigma} \vert_{G(\bar{B})}.$$
Now define $\sigma_U \colon G(A^{\circ}, J) \to U$ to be $\nu_U^{-1} \circ (\bar{\sigma} \vert_{G(A^{\circ}, J)}).$ Then $$
\sigma = (\sigma_U, \sigma_S) \colon G(A^{\circ}) = G(A^{\circ}, J) \rtimes G(\bar{B}) \to H^{\circ}
$$
is again a morphism of algebraic groups satisfying $\nu \circ \sigma = \bar{\sigma}.$

Thus, for the morphisms $\sigma$ constructed in both cases, it
now follows from Proposition \ref{P:FI-5} that $\nu \circ \sigma \circ \pi_{A^{\circ}} = \nu \circ \tilde{\sigma}.$ Hence $\chi \colon \tilde{G}(A^{\circ}) \to H^{\circ}$ defined by
$$
\chi (g) = \tilde{\sigma}(g)^{-1} \cdot (\sigma \circ \pi_{A^{\circ}}) (g)
$$
has values in $Z(H^{\circ}).$ This, in conjunction with the fact that $\tilde{\sigma}$ and $\sigma \circ \pi_{A^{\circ}}$ are group homomorphisms, implies that $\chi$ is also a group homomorphism. However, since $\tilde{G}(A^{\circ})$ coincides with its commutator (\cite{St1}, Corollary 4.4), $\chi$ must be trivial, and therefore
$$
\sigma \circ \pi_{A^{\circ}} = \tilde{\sigma},
$$
as required.

\end{proof}

Thus, combining the results of Propositions \ref{P:FI-3}, \ref{P:FI-5}, \ref{P:R-1}, and \ref{P:R-2}, we obtain the following
\vskip2mm

\begin{thm}\label{T:R-1}
Let $\Phi$ be a reduced irreducible root system of rank $\geq 2$, $R$ a commutative ring such that $(\Phi, R)$ is a nice pair, and $K$ an algebraically closed field. Assume that $R$ is noetherian if char $K > 0.$ Furthermore let $G$ be the universal Chevalley-Demazure group scheme of type $\Phi$ and let $\rho \colon G(R)^+ \to GL_n (K)$ be a finite-dimensional linear representation of the elementary subgroup $G(R)^+ \subset G(R)$ over $K.$ Set $H = \overline{\rho (G(R)^+)}$ (Zariski closure), and let $H^{\circ}$ be the connected component of the identity of $H$. Then in each of the following situations
\vskip1mm

\noindent {\rm (1)} $H^{\circ}$ is reductive;

\vskip1mm

\noindent {\rm (2)} char $K = 0$ and $R$ is semilocal;

\vskip1mm

\noindent {\rm (3)} char $K = 0$ and $H^{\circ}$ satisfies condition {\rm (Z)}

\vskip1mm

\noindent there exists a commutative finite-dimensional $K$-algebra $B$, a ring homomorphism $f \colon R \to B$ with Zariski-dense image and a morphism $\sigma \colon G(B) \to H$ of algebraic $K$-groups such that for a suitable subgroup $\Gamma \subset G(R)^+$ of finite index we have
$$
\rho \vert_{\Gamma} = (\sigma \circ F) \vert_{\Gamma},
$$
where $F \colon G(R)^+ \to G(B)^+$ is the group homomorphism induced by $f$.

\end{thm}

\vskip1mm

Recall that we showed in Proposition \ref{P:FI-2} that if char $K = 0$ and the unipotent radical $U$ of $H^{\circ}$ is commutative, then $H^{\circ}$ satisfies (Z). Hence, Theorem \ref{T:R-1} yields all of the assertions of the Main Theorem.

\vskip2mm

\noindent {\bf Example 6.8.} If $H^{\circ}$ is reductive, then it follows from Lemma \ref{L:FI-1} and our construction that $B$ can be chosen to have trivial Jacobson radical, and therefore $B \simeq K \times \cdots \times K$ ($r$ copies). Then the homomorphism $f \colon R \to B$ in Theorem \ref{T:R-1} is of the form $f = (f_1, \dots, f_m)$, where each component is a homomorphism
$f_i \colon R \to K.$ In particular, if $R = \Z [x_1, \dots, x_k]$, then each $f_i$ is just a specialization map. So, in this case, we obtain from Theorem \ref{T:R-1} that
any representation $\rho \colon G(R)^+ \to GL_n (K)$ coincides on a subgroup of finite index $\Gamma \subset G(R)^+$ with $\sigma \circ F$, where $F = (F_1, \dots, F_r)$ and each component $F_i \colon G(R)^+ \to G(K)$ is induced by a specialization homomorphism, and $\sigma \colon G(K) \times \cdots \times G(K) \to H$ is a morphism of algebraic $K$-groups. Thus, Theorem \ref{T:R-1} generalizes the result of Shenfeld \cite{Sh} which treats the case $G= SL_n,$ $R = \Z[x_1, \dots, x_k]$, and $K= \C,$ using the centrality of the congruence kernel of $G = SL_n (R)$ established in \cite{KN} and mimicking the argument of Bass-Milnor-Serre \cite{BMS}.

\vskip2mm

\noindent {\bf Example 6.9.} Now assume that the unipotent radical $U$ of $H^{\circ}$ is commutative and that char~$K = 0.$ Then it follows from Lemma \ref{L:FI-1} and our construction that one can choose $B$ so that its Jacobson radical $J = J(B)$ satisfies $J^2 = \{ 0 \}.$ Moreover, it follows from Proposition \ref{P:R-3}(ii) that we can write
$$
G(B) = G(B_1) \times \cdots \times G(B_r),
$$
where each $B_i$ is a finite dimensional local $K$-algebra of the form $B_i = K \oplus J_i$ with $J_i^2 = \{ 0 \}.$ Hence it is enough to analyze the case where $B = K \oplus J$ with $J^2 = \{ 0 \}.$
%Now by the Wedderburn-Malcev theorem (\cite{P}, Theorem 11.6), there exists a subalgebra $\bar{B} \subset B$ such that $B = \bar{B} \oplus J,$ and then $\bar{B} \simeq B/J \simeq K \times \cdots \times K$ ($m$ factors). Let $e_i = (0, \cdots, 0, 1, 0, \cdots, 0) \in \bar{B}$ be the $i$th standard basis vector. Then
%$$
%B = \bigoplus_{i=1}^m B_i, \ \ \ {\rm where} \ \ B_i = e_i B.
%$$
%Clearly $B_i = e_i \bar{B} \oplus J_i = K \oplus J_i$, where $J_i = e_i J.$ Since $G(B) \simeq G(B_1) \times \cdots \times G(B_m),$ it is enough to analyze the case where $m = 1,$ i.e. $B = K \oplus J$ and $J^2 = \{ 0 \}.$
So now choose a $K$-basis $\{ v_1, \dots, v_d \}$ of $J$. Then a homomorphism $f \colon R \to B$ can be written in the form
$$
f(r) = (f_0 (r), f_1(r) v_1 + \cdots + f_d (r) v_d),
$$
where $f_0 \colon R \to K$ is a ring homomorphism and the $f_i$'s, for $i \geq 1$, satisfy $$
f_i (r_1 r_2) = f_0 (r_1) f_i (r_2) + f_i (r_1) f_0 (r_2).
$$
Thus, each $f_i,$ $i \geq 1,$ is a derivation (with respect to $f_0$), and we recover, in a slightly different form, the result of L.~Lifschitz and A.~Rapinchuk \cite{LR}, which was established when $R = k$ is a field of characteristic zero (as we observed earlier, in this case $G(R)$ does not have proper noncentral normal subgroups, hence $\Gamma = G(R)$).

\vskip2mm

\noindent {\bf Remark 6.10.} Borel and Tits \cite{BT} consider abstract homomorphisms into groups of points over not necessarily algebraically closed fields. It appears that our results can also be generalized to representations over non-algebraically closed fields. However, this will require an analysis of algebraic rings over non-algebraically closed fields and will be given elsewhere, along with the verification of condition (Z) in some new cases.

\vskip5mm

\bibliographystyle{amsplain}

\end{document}